\documentclass{amsart}

\usepackage{graphicx} % Required for inserting images
\usepackage{todonotes}
\usepackage[margin=1.3in]{geometry}
\usepackage[english]{babel}
\usepackage{hyperref,amssymb,amsmath,amsthm,color,mathtools,tikz,subfig,dirtytalk,csquotes,soul,orcidlink,cleveref}
\usetikzlibrary{arrows}

\numberwithin{equation}{section}

\newtheorem{example}{Example}[section]

\newtheorem{dhef}[example]{Definition}
\newtheorem{lemma}[example]{Lemma}
\newtheorem{remark}[example]{Remark}
\newtheorem{prop}[example]{Proposition}
\newtheorem{theo}[example]{Theorem}
\newtheorem{cor}[example]{Corollary}

\def\elle#1{L^{#1}}
\def\elleom#1{L^{#1}(\Omega)}
\def\elleomT#1{L^{#1}(\Omega_T)}

\def\Sobom#1{W^{1,#1}_0(\Omega)}

\def\H{H^1_0}
\def\Hom{H^1_0(\Omega)}

\def\io{\int_\Omega}
\def\iot{\int_{\Omega_t}}
\def\ioT{\int_{\Omega_T}}
\def\norma#1#2{ \|#1 \|_{#2}}

\def\N{\mathbb{N}}

\def\R{\mathbb{R}}

\newcommand{\abs}[1]{\left|#1\right|}

\DeclareRobustCommand{\rchi}{{\mathpalette\irchi\relax}}
\newcommand{\irchi}[2]{\raisebox{\depth}{$#1\chi$}} % inner command, used by \rchi

\title{Existence and regularity of solutions to parabolic-elliptic nonlinear systems}

\author{Marco Picerni
\orcidlink{0009-0004-4364-4831}}\email{mpicerni@sissa.it}
\address{SISSA, via Bonomea 265, 34136, Trieste, Italy}

\begin{document}

\begin{abstract}
    In this paper, we study the existence and summability of the solutions to the following parabolic-elliptic system of partial differential equations with discontinuous coefficients:
\begin{equation*}
    \left\{\begin{array}{cc}
    u_t-\operatorname{div}(A(x, t) \nabla u)=-\operatorname{div}(u M(x) \nabla \psi)+f(x,t) & \text { in } \Omega_T, \\
    -\operatorname{div}(M(x) \nabla \psi)=|u|^\theta & \text { in } \Omega_T, \\
    \psi(x, t)=0 & \text { on } \partial\Omega\times(0, T), \\
    u(x, t)=0 & \text { on } \partial\Omega\times(0, T), \\
    u(x, 0)=0 & \text { in } \Omega
    \end{array}\right.
\end{equation*}
Here $\Omega$ is an open and bounded subset of $\R^N$, $N>2$, $\theta\in(0,\frac{2}{N})$ and $0<T<+\infty$. 

We prove existence results for data $f\in\elleomT 1$ and a corresponding increase in summability that obeys the $L^p$-regularity theorems for parabolic equations found in \cite{Aronson-Serrin} and \cite{BDGO-parabolic}. Moreover, despite the term $u M(x)\nabla\psi$ not being regular enough, the solution $u$ belongs to $\elleomT s\cap \elle q(0,T;\Sobom q)$ for suitable $s>1$ and $q>1$.
\end{abstract}

\maketitle

\section{Introduction and main results}

\subsection{Introduction}

Consider the following parabolic-elliptic system with zero Dirichlet boundary conditions
\begin{equation}\label{eq:KS-Sys-ParabEll}
    \left\{\begin{array}{cc}
    u_t-\operatorname{div}(A(x, t) \nabla u)=-\operatorname{div}(u M(x) \nabla \psi)+f(x,t) & \text { in } \Omega_T, \\
    -\operatorname{div}(M(x) \nabla \psi)=|u|^\theta & \text { in } \Omega_T, \\
    \psi(x, t)=0 & \text { on } \partial\Omega\times(0, T), \\
    u(x, t)=0 & \text { on } \partial\Omega\times(0, T), \\
    u(x, 0)=0 & \text { in } \Omega.
    \end{array}\right.
\end{equation}
Here, $N\geq3$, $\Omega$ is an open and bounded subset of $\R^N$, $T>0$ and $\Omega_T=\Omega\times(0,T)$. We assume that the matrix-valued measurable functions $A(x,t)$ and $M(x)$ are bounded and uniformly elliptic, that is:
$$\alpha\abs{\xi}^2\leq A(x,t)\xi\cdot\xi\leq \beta\abs{\xi}^2 \quad\forall \xi\in\R^N, \quad\text{for a.e. } (x,t)\in\Omega_T,$$
$$\alpha\abs{\xi}^2\leq M(x)\xi\cdot\xi\leq \beta\abs{\xi}^2 \quad\forall \xi\in\R^N, \quad\text{for a.e. } x\in\Omega,$$
for some $\beta\geq\alpha>0$. Moreover, $f$ is a nonnegative function which belongs to $\elleomT1$, and $\theta\in(0,\frac2N)$.

System \eqref{eq:KS-Sys-ParabEll} is loosely related to the minimal version of the Keller-Segel system \cite{UsersGuideChemo}, which is a fully parabolic system used to describe chemotaxis. Our system is a discontinuous-coefficient version of such model with an added source term $f\geq 0$ and zero Dirichlet boundary conditions. Furthermore, the second equation of our system is elliptic and the dependence on $u$ in the second term is sublinear (recall that $N>2$, which implies $\theta\leq\frac2N<1$), while the original model features a linear dependence on $u$ in the second equation.

This system has been studied by Boccardo, Orsina and Porzio in \cite{KSParabEll} (see also \cite{el_moutaouakil_numerical_2025}) with $f=0$ and Dirichlet boundary conditions
\begin{equation}
    \left\{\begin{array}{cc}
    \psi(x, t)=0 & \text { on } \partial\Omega\times(0, T), \\
    u(x, t)=0 & \text { on } \partial\Omega\times(0, T), \\
    u(x, 0)=u_0(x) & \text { in } \Omega,
    \end{array}\right.
\end{equation}
for any initial datum $u_0\in\elleom1$. 
Under these assumptions, the authors proved the existence of distributional solutions $(u,\psi)$ such that $u\in \elle\infty(t,T;\elleom q)$ for every $q\in(1,+\infty)$ and $t>0$, and satisfying the bound (also known as \emph{supercontractive estimate} \cite{porzio_decay_2009,porzio_uniform_2015})
\begin{equation}
    \norma{u(t)}{\elleom q}\leq C_q\,\frac{\norma{u_0}{\elleom1}}{t^{\frac{N}{2}\left(1-\frac1q\right)}}.
\end{equation}
Note that such estimate is very close to the one in place for solutions of the heat equation with $\elleom1$ data, which is
\[
\norma{u(t)}{\elleom q}\leq C\,\frac{\norma{u_0}{\elleom1}}{t^{\frac{N}{2}}}.
\]
On the other hand, it was proved in \cite{BOPorretta} that, for equations of the type 
\begin{equation}
    u_t-\operatorname{div}(A(x, t) \nabla u)=-\operatorname{div}(u E(x,t))\quad \text{with } E\in\left(\elleomT2\right)^N,
\end{equation}
it is only possible to find a solution $u\in\elleomT1$ such that 
\[\log(1+\abs{u})\in\elle2(0,T;\Hom)\]
(see also \cite{BoccardoJDE} for a similar result in the elliptic case).
This, together with the fact that the vector field $E(x,t)=M(x) \nabla \psi(x,t)$ only belongs to $\left(\elleomT2\right)^N$, indicates that the interplay between the two equations in \eqref{eq:KS-Sys-ParabEll} allows to improve the summability that the first equation would lead to.

A similar phenomenon was observed by Boccardo and Orsina in \cite{BO-Kellersegel}, where the stationary version of \eqref{eq:KS-Sys-ParabEll}
\begin{equation}\label{eq:intro:KS:elliptic}
    \left\{\begin{array}{cc}
    -\operatorname{div}(A(x) \nabla u)+u=-\operatorname{div}(u M(x) \nabla \psi)+f(x) & \text { in } \Omega, \\
    -\operatorname{div}(M(x) \nabla \psi)=|u|^\theta & \text { in } \Omega, \\
    \psi=0 & \text { on } \partial\Omega, \\
    u=0 & \text { on } \partial\Omega
    \end{array}\right.
\end{equation}
was studied. Here, the interplay between the two equations allows to prove the existence of solutions which obey the $\elle p$-regularity theorems proved by Stampacchia and Boccardo--Gallou\"et \cite{BoccardoGallouet,StampacchiaElliptic} (note that this is a great improvement over the summability attained by considering the first equation alone \cite{BoccardoJDE}). Further analyses of systems related to \eqref{eq:intro:KS:elliptic} have been carried out in  \cite{boccardo2022EllipticChemotaxisSystem,boccardo2026ExistenceResultsNonlinear,boccardo_elliptic_2023}.

\subsection{Main results}

Motivated by \cite{BO-Kellersegel,KSParabEll}, we establish existence and summability results for solutions $(u,\psi)$ to \eqref{eq:KS-Sys-ParabEll}. The notion of solution and the summability attained depend on the integrability of the datum $f$.

Throughout the paper, $p^\star$ and $p^{\star\star}$ will denote the real numbers which satisfy (for $p<N+2$ and $p<\frac{N+2}{2}$, respectively) 
$$\frac{1}{p^\star}=\frac{1}{p}-\frac{1}{N+2}\quad\text{ and }\quad \frac{1}{p^{\star\star}}=\frac{1}{p}-\frac{2}{N+2}.$$ 
This choice of notation is suggested by the results of Aronson--Serrin \cite{Aronson-Serrin} and Boccardo--Dall'Aglio--Gallou\"et--Orsina \cite{BDGO-parabolic}, where $L^p$ regularity results for parabolic equations have been established, and the fact that these results are formally equivalent to the ones for elliptic equations (up to writing $N+2$ instead of $N$ in the Sobolev exponents).

This paper contains four main theorems, which are ordered depending on the summability requirements on the datum $f$ (the first being the one with the highest requirements and the last dealing with the case $f\in\elleomT1$). We begin by stating Theorem \ref{Main_Theorem_1}, establishes existence of bounded weak solutions to \eqref{eq:KS-Sys-ParabEll}.

\begin{theo}\label{Main_Theorem_1}
    Let $f\in\elle m(\Omega_T)$ with $m>\frac{N+2}{2}$, then there exist $u,\psi$ which solve \eqref{eq:KS-Sys-ParabEll}. More precisely, there exist 
    $$u\in C^0([0,T];\elleom2)\cap\elle2(0,T;\Hom)\cap\elle\infty(\Omega_T)\text{ such that }u_t\in\elle2(0,T;H^{-1}(\Omega))$$
     and 
    $$\psi\in C^0([0,T];\Hom)\cap\elle\infty(\Omega_T)$$ such that $u\geq0$, $\psi\geq0$ and, for all $\varphi\in\elle2(0,T;\Hom)$ and all $v\in\Hom$,
    \begin{equation}\label{DefSol_parabell_theo1}
        \begin{dcases}
                \int_0^T {}_{H^{-1}}\langle u_t,\varphi\rangle_{\H}+\ioT A(x,t)\nabla u\cdot \nabla \varphi = \ioT uM(x)\nabla \psi\cdot\nabla\varphi + \ioT f(x,t)\varphi\\
                \io M(x)\nabla \psi(t)\cdot\nabla v = \io u(t)^\theta v\quad\forall\,t\in(0,T).
        \end{dcases}
    \end{equation}
    Moreover, the initial condition is satisfied in the strong sense for both $u$ and $\psi$, that is $u(0)=0$ and $\psi(0)=0$ almost--everywhere in $\Omega$.
\end{theo}

If the summability of $f$ is not enough to guarantee the boundedness of the solution $u$, we prove that the summability of $u$ depends on the one of $f$, accordingly with the results of \cite{BDGO-parabolic}.

\begin{theo}\label{Main_Theorem_2}
    Let $f\in\elle m(\Omega_T)$ with $\frac{2N+4}{N+4}\leq m <\frac{N+2}{2}$, then there exist 
    $$u\in\elle\infty(0,T;\elleom{\frac{Nm}{N+2-2m}})\cap\elle2(0,T;\Hom)\cap \elleomT {m^{\star\star}}$$ 
    such that $u_t\in\elle1(0,T;W^{-1,1}(\Omega))$ 
    and $$\psi\in C^0([0,T];\Hom)\cap\elle\infty(\Omega_T)$$ such that $u\geq0$, $\psi\geq0$ and, for all $\varphi\in\elle\infty(0,T;\Sobom \infty)$ and all $v\in\Hom$,
    \begin{equation}\label{DefSol_parabell_theo2}
        \begin{dcases}
                \int_0^T \langle u_t,\varphi\rangle+\ioT A(x,t)\nabla u\cdot \nabla \varphi = \ioT uM(x)\nabla \psi\cdot\nabla\varphi + \ioT f(x,t)\varphi\\
                \io M(x)\nabla \psi(t)\cdot\nabla v = \io u(t)^\theta v\quad\forall\,t\in(0,T).
        \end{dcases}
    \end{equation}
    
    Finally, the initial condition for $u$ is satisfied in the following weak sense: there exists a sequence $\{u_n\}_n$ of functions in $C^0([0, T] ; L^2(\Omega)) \cap L^2(0, T ;\Hom)$ converging to $u$ in $\elleomT1$ such that $u_n(0)=0$, and $\psi(0)=0$ almost--everywhere in $\Omega$.
\end{theo}
\begin{remark}
    The choice of test functions belonging to $\elle\infty(0,T;\Sobom \infty)$ may seem somewhat restrictive since, arguably, the distributions that appear in the equation act on larger function spaces. Our choice, however, has the advantage of simplifying the notation while remaining in a Sobolev space context. Moreover, since the aforementioned larger spaces are of the form $\elle a(0,T;\elleomT b)$, one can recover the complete set of test functions arguing by density. 
\end{remark}
To prove the existence of a distributional solution with smooth test functions (such as in Theorem \ref{Main_Theorem_2}), we need the term $uM(x)\nabla\psi$ to belong to $\elleomT1$. For this to be possible, since we only know that $\nabla\psi$ belongs to $\left(\elleomT2\right)^N$, it is necessary for $u$ to belong to (at least) $\elleomT2$. Such requirement is equivalent to
$$m^{\star\star}=\frac{(N+2)m}{N+2-2m}\geq2,$$
which implies that $m$ shall be greater than $\frac{2N+4}{N+6}$ (recall that $m^{\star\star}$ is only defined for $m<\frac{N+2}{2}$). We thus have the following result

\begin{theo}\label{Main_Theorem_3}
     Let $f\in\elleomT m$ with $\frac{2N+4}{N+6}\leq m<\frac{2N+4}{N+4}$, then there exist 
     $$u\in\elle\infty(0,T;\elleom{\frac{Nm}{N+2-2m}})\cap\elle {m^\star}(0,T;\Sobom{m^\star})\cap \elleomT {m^{\star\star}}$$ such that $u_t\in\elle1(0,T,W^{-1,1}(\Omega))$ and $\psi\in C^0([0,T];\Hom)\cap\elle\infty(\Omega_T)$ such that $u\geq0$, $\psi\geq0$ and for all $\varphi\in\elle\infty(0,T;\Sobom \infty)$ and all $v\in\Hom$
    % $\varphi\in\elle2(0,T;\Sobom N)$ with $\varphi_t\in\elle2(0,T,W^{-1,\frac{N}{N-1}})$
    \begin{equation}\label{DefSol_parabell_theo3}
        \begin{dcases}
                \int_0^T \langle u_t,\varphi\rangle+\ioT A(x,t)\nabla u\cdot \nabla \varphi = \ioT uM(x)\nabla \psi\cdot\nabla\varphi + \ioT f(x,t)\varphi\\
                \io M(x)\nabla \psi(t)\cdot\nabla v = \io u(t)^\theta v\quad\forall\,t\in(0,T).
        \end{dcases}
    \end{equation}

    Finally, the initial condition for $u$ is satisfied in the following weak sense: there exists a sequence $\{u_n\}_n$ of functions in $C^0([0, T] ; L^2(\Omega)) \cap L^2(0, T ;\Hom)$ converging to $u$ in $\elleomT1$ such that $u_n(0)=0$, and $\psi(0)=0$ almost--everywhere in $\Omega$.
\end{theo}

For highly singular data, since $uM(x)\nabla\psi$ does not belong to $\elleomT1$, we cannot prove the existence of a distributional solution and we will need to use entropy solutions. The definition of entropy solutions was introduced in \cite{LavoroB6} for elliptic problems and in \cite{Prignet-entropy} for parabolic problems. We will consider solutions to the form $(u,\psi)$ where $u$ is an entropy solution to the first equation of \eqref{eq:KS-Sys-ParabEll} and $\psi$ is a weak solution to the second equation.

\noindent From now on, we will use the notation 
\[
T_k(s)=
\begin{cases}
    k\quad\text{if }\,s\geq k\\
    s\quad\text{if }\,\abs{s}< k\\
    -k\quad\text{if }\,s\leq- k\\
    \end{cases}\quad\text{and}\quad
    G_k(s)=s-T_k(s).
\]

\begin{dhef}\label{Def:entropysol}
    The couple $(u,\psi)$ is an \textbf{entropy solution} to \eqref{eq:KS-Sys-ParabEll} if:
    \begin{itemize}
        \item $u\in\elleomT1$;
        \item $T_k(u)\in\elle2(0,T;\Hom)$ for every $k>0$;
        \item $\psi\in C^0([0,T],\Hom)$;
        \item $\psi(0)=0$ almost--everywhere in $\Omega$;
        \item $u(0)=0$ in the sense that there exists a sequence $\{u_n\}_n$ of functions in $C^0([0, T] ; L^2(\Omega)) \cap L^2(0, T ;\Hom)$ converging to $u$ in $\elleomT1$ such that $u_n(0)=0$ for every $n\in\N$.
    \end{itemize}
    and, for every $\varphi\in\elleomT\infty\cap\elle2(0,T;\Hom)$ such that $\varphi_t\in \elleomT 1+ \elle2(0,T;H^{-1}(\Omega))$ and $\varphi(0)=0$, every $v\in\Hom$, almost every $t\in(0,T)$ and every $k\geq0$, we have
    \begin{equation}
        \begin{dcases}
                \io \Theta_k(u-\varphi)(t)+\int_0^t \langle \varphi_t,T_k(u-\varphi)\rangle+
                \iot A(x,s)\nabla u\cdot \nabla T_k(u-\varphi) \leq
                \\ \quad\iot u M(x)\nabla \psi\cdot\nabla T_k(u-\varphi) + \iot f(x,s)T_k(u-\varphi)\\
                \io M(x)\nabla \psi(t)\cdot\nabla v = \io u(t)^\theta v.
        \end{dcases}
    \end{equation}
    Here $\Theta_k(s)$ denotes the primitive of $T_k(s)$ that is zero for $s=0$.
\end{dhef}

\begin{theo}\label{Main_Theorem_4}
    Let $f\in\elleomT m$ with $1\leq m<\frac{2N+4}{N+6}$, then there exists an entropy solution $(u,\psi)$ to \eqref{eq:KS-Sys-ParabEll} such that:
    \begin{itemize}
        \item if $m>1$, $u\in\elle\infty(0,T;\elleom1)\cap\elleomT {m^{\star\star}} \cap \elle {m^\star}(0,T;\Sobom {m^\star})$;
        \item if $m=1$, $u\in\elle\infty(0,T;\elleom1)\cap\elleomT s \cap \elle q(0,T;\Sobom q)$ for every $s<1^{\star\star}$ and $q<1^\star$;
    \end{itemize}
    and $\psi\in C^0([0,T],\Hom)\cap\elleomT\infty$. Moreover, $u\geq0$ and $\psi\geq0$.
\end{theo}

\section{Proof of the main results}
{To prove the existence and regularity of solutions, we proceed by approximation. To this end, we consider the following approximating system:

\begin{equation}\label{eq:approx_KS_system}
    \left\{\begin{array}{cc}
    (u_n)_t-\operatorname{div}(A(x, t) \nabla u_n)=-\operatorname{div}(T_n(u_n) M(x) \nabla \psi_n)+T_n(f(x,t)) & \text { in } \Omega_T, \\
    -\operatorname{div}(M(x) \nabla \psi_n)=|T_n(u_n)|^\theta & \text { in } \Omega_T, \\
    \psi_n(x, t)=0 & \text { on } \partial\Omega\times(0, T), \\
    u_n(x, t)=0 & \text { on } \partial\Omega\times(0, T), \\
    u_n(x, 0)=0 & \text { in } \Omega.
    \end{array}\right.
\end{equation}

In \Cref{sec:existenceapprox}, we prove the existence of a weak solution $(u_n,\psi_n)$ to \eqref{eq:approx_KS_system} for every $n\in\N$. Then, depending on the summability of the datum $f$, we prove that one can pass to the limit as $n\to+\infty$ under suitable a priori bounds on $(u_n,\psi_n)_n$ to obtain a solution to \eqref{eq:KS-Sys-ParabEll}. In particular, \Cref{sec:highsummability} contains the proof of Theorem \ref{Main_Theorem_1} and Theorem \ref{Main_Theorem_2}, \Cref{sec:distrib_sol} contains the proof of Theorem \ref{Main_Theorem_3}, and \Cref{sec:entropy} contains the proof of Theorem \ref{Main_Theorem_4}.
}

\subsection{Existence of an approximate solution}\label{sec:existenceapprox}

\begin{theo}\label{theo:exists_approx_solutions}
    Let $f\in\elleomT1$. Then there exist $u_n\in \elle2(0,T;\Hom)$ such that $(u_n)_t\in \elle2(0,T;H^{-1}(\Omega))$ and $\psi_n\in C^0([0,T];\Hom)\cap\elleomT\infty$ which solve \eqref{eq:approx_KS_system} in the sense that, for all $\varphi\in\elle2(0,T;\Hom)$ with $\varphi_t\in\elle2(0,T;H^{-1}(\Omega))$ and all $v\in\Hom$,
    \begin{equation}
        \begin{dcases}
                \int_0^T {}_{H^{-1}}\langle (u_n)_t,\varphi\rangle_{\H}+\ioT A(x,t)\nabla u_n\cdot \nabla \varphi = \ioT T_n(u_n)M(x)\nabla \psi_n\cdot\nabla\varphi + \ioT T_n(f)\varphi\\
                \io M(x)\nabla \psi_n(t)\cdot\nabla v = \io \abs{T_n(u_n)}^\theta v\quad\forall\,t\in(0,T).
        \end{dcases}
    \end{equation}
    and satisfy the boundary conditions in \eqref{eq:approx_KS_system}.
\end{theo}
\begin{remark}
    To simplify the notation, we will write $\langle (u_n)_t,\varphi\rangle$ instead of ${}_{H^{-1}}\langle (u_n)_t,\varphi\rangle_{\H}$ during computations.
\end{remark}
\begin{remark}
    By choosing test functions of the type $\rho(x)\eta(t)$ and arguing by density, it can be shown that an equivalent definition of solution to the first equation to \eqref{eq:approx_KS_system} is
    \begin{equation}
        {}_{H^{-1}}\langle (u_n)_t,\varphi\rangle_{\H}+\io A(x,t)\nabla u_n\cdot \nabla \varphi = \io T_n(u_n)M(x)\nabla \psi_n\cdot\nabla\varphi + \io T_n(f)\varphi
    \end{equation}
    for a.e. $t\in(0,T)$.
\end{remark}

\begin{remark}
    The function $u_n$ appearing in Theorem \ref{theo:exists_approx_solutions} belongs to $C^0([0,T];\elleom2)$ (see \cite[Section 5.9.2]{EvansPDE}).
\end{remark}

\begin{proof}[Proof of Theorem \ref{theo:exists_approx_solutions}]
    We start by fixing $w\in C^0([0,T];\elleom2)$ such that $w(x,0)=0$ and considering the system
    \begin{equation}\label{eq:KSParabApproxSchauder}
        \left\{\begin{array}{cc}
            v_t-\operatorname{div}(A(x, t) \nabla v)=-\operatorname{div}(T_n(w) M(x) \nabla \eta)+T_n(f(x,t)) & \text { in } \Omega_T, \\
            -\operatorname{div}(M(x) \nabla \eta)=|T_n(w)|^\theta & \text { in } \Omega_T, \\
        \end{array}\right.
    \end{equation}
    with zero boundary conditions. We claim that there exists a unique choice of $v\in\elle2(0,T;\Hom)$ such that $v_t\in\elle2(0,T;H^{-1}(\Omega))$ and $\eta\in\elle\infty(\Omega_T)\cap\elle\infty(0,T;\Hom)$ which solve \eqref{eq:KSParabApproxSchauder}.
    
    The existence of a solution to the second equation follows from Lax Milgram's theorem. Moreover (see \cite[Th\'eor\`eme 4.1]{StampacchiaElliptic}), there exists a constant $C>0$ such that
    \begin{equation}
        \norma{\eta(t)}{\Hom}\leq C n^\theta\qquad\text{and}\qquad\norma{\eta(t)}{\elleom\infty}\leq C n^\theta\qquad \forall t\in [0,T].
    \end{equation}
    Thus, $\eta\in C^0([0,T];\Hom)$.
    Moreover, by linearity and \cite[Th\'eor\`eme 4.1]{StampacchiaElliptic}, we have (up to increasing the constant $C$)
    $$\norma{\eta(t)-\eta(t')}{\Hom}
    \leq
    C\norma{w(t)-w(t')}{\elleom2}^\theta,$$
    and the right hand side term goes to zero as $\abs{t-t'}\to 0$.

    We now focus on the first equation: observe that
    \begin{equation}
        \left\{\begin{array}{cc}
           \abs{T_n(w) M(x) \nabla \eta}\leq n\beta\abs{\nabla \eta}\in\elle\infty(0,T;\elleom2)\\
           |T_n(f)|\leq n.\\
        \end{array}\right.
    \end{equation}
    This implies (see, for example, \cite{Lions}) that there exists a unique solution $v\in\elle2(0,T;\Hom)$ to the first equation such that $v(0)=0$. Such solution also satisfies the condition $v_t\in\elle2(0,T;H^{-1}(\Omega))$ and, since (up to increasing the constant $C$)
    $$\norma{T_n(w) M(x) \nabla \eta}{\elle\infty(0,T;\elleom2)}+\norma{T_n(f)}{\elle\infty(\Omega_T)}\leq Cn^{\theta+1},$$
    we have the estimate
    \begin{equation}\label{BoundSchauderParabolico}
        \norma{v}{\elle2(0,T;\Hom)}\leq Cn^{\theta+1}\quad \text{and}\quad \norma{v_t}{\elle2(0,T;H^{-1}(\Omega))}\leq Cn^{\theta+1}.
    \end{equation}
    Observe that estimate \eqref{BoundSchauderParabolico} implies (see \cite[Section 5.9.2]{EvansPDE}) that $v\in C^0([0,T];\elleom2)$ and 
    \begin{equation}\label{BoundSchauderParabolicoBis}
        \norma{v}{C^0([0,T];\elleom2)}\leq Cn^{\theta+1}.
    \end{equation}

    We now want to apply Schauder's fixed point theorem to prove the existence of a solution to \eqref{eq:approx_KS_system}. Consider the nonlinear operator
    \begin{align*}
         S: C^0([0,T];\elleom2) & \to C^0([0,T];\elleom2)\\
        w&\longmapsto v.
    \end{align*}
    By \eqref{BoundSchauderParabolicoBis}, we know that the closed ball of $C^0([0,T];\elleom2)$ of radius $R=C n^{\theta+1}$ is invariant for $S$. Moreover, by \eqref{BoundSchauderParabolicoBis} and \cite[Corollary 4]{Simon}, $S$ is a compact map.
    
    We now prove that $S$ is continuous. Consider a sequence $(w_k)_k\subset C^0([0,T];\elleom2)$ such that $w_k$ strongly converges to some $w\in C^0([0,T];\elleom2)$. Let $\eta_k$ and $\eta$ be the corresponding solutions to the second equation to \eqref{eq:KSParabApproxSchauder} (recall that both $\eta$ and the $\eta_k$'s belong to $C^0([0,T];\Hom)$). Then, by linearity, we have
    $$\norma{\eta_k(t)-\eta(t)}{\Hom}\leq C\norma{w_k(t)-w(t)}{\elleom2}^\theta\leq C\norma{w_k-w}{C^0([0,T];\elleom2)}^\theta$$
    for every $t\in(0,T)$. Taking the supremum over $t$, it follows that $\eta_k\to\eta$ in $C^0([0,T];\elleom2)$.
    
    By the compactness of $S$, we know that any subsequence of $S(w_k)$ has a subsequence which converges in $C^0([0,T];\elleom2)$ to some $\tilde v$ (which may depend on the chosen subsequence). At the same time, we have that
    \begin{equation}
        \left\{\begin{array}{cc}
          T_n(w_k)\to T_n(w)\text{ in } C^0([0,T];\elleom p) \text{ for any } p\in(1,\infty) \text{ by Lebesgue's theorem } \\
        \nabla\eta_k\to\nabla\eta\text{ in } C^0([0,T];\elleom 2).
        \end{array}\right.
    \end{equation}
    Hence, $\tilde v$ satisfies the equation
    $$
    \tilde v_t-\operatorname{div}(A(x, t) \nabla\tilde  v)=-\operatorname{div}(T_n(w) M(x) \nabla \eta)+T_n(f(x,t)).
    $$
    and, since $S(w_k)(0)=0$ for every $k$, $\tilde v(0)=0$.
    This implies, by uniqueness, that $\tilde v= S(w)$ and thus $S$ is continuous. We conclude the proof by invoking Schauder's fixed point theorem.
\end{proof}

We now prove several properties of the approximate solutions $(u_n,\psi_n)$. 
First, we prove that the positivity of the datum $f$ implies the positivity of the solutions.

\begin{lemma}
    The solutions $u_n$, $\psi_n$ to \eqref{eq:approx_KS_system} are nonnegative.
\end{lemma}
\begin{proof}
    We will use the same technique used in the proof of the previous lemma.
    First, we choose $\psi_n\rchi_{\{\psi_n\leq0\}}$ as a test function in the second equation to obtain
    $$
    \alpha\io\abs{\nabla\psi_n^-}^2=
    \alpha\int_{\psi_n\leq0}\abs{\nabla\psi_n}^2\leq 
    \int_{\psi_n\leq0}\abs{T_n(u_n)}^\theta\psi_n\leq0
    $$
    which, applying Poincaré's inequality, implies that $\psi_n^-=0$.

    We now prove that $u_n\geq0$: choosing $T_h(u_n\rchi_{\{u_n\leq0\}})$ as a test function, we get (using Young's inequality)
    $$
    \io (u_n)_tT_h(u_n\rchi_{\{u_n\leq0\}}) + \alpha \io \abs{\nabla T_h(u_n\rchi_{\{u_n\leq0\}})}^2
    \leq$$
    $$
    \frac{1}{2\alpha}\int_{-h\leq u_n\leq0} T_n(u_n)^2 \abs{M(x)\nabla\psi_n}^2 + \frac{\alpha}{2} \io \abs{\nabla T_h(u_n\rchi_{\{u_n\leq0\}})}^2 + \io f(x,t)T_h(u_n\rchi_{\{u_n\leq0\}}).
    $$
    After cancellations (notice that the last term is negative, since $f\geq0$), this leads us to
    $$
    \io (u_n)_tT_h(u_n\rchi_{\{u_n\leq0\}})
    \leq
    \frac{\beta^2 h^2}{2\alpha} \int_{-h\leq u_n\leq0} \abs{\nabla\psi_n}^2.
    $$
    We now divide by $h$ and take the limit $h\to0$, obtaining
    $$
    \frac{d}{dt}\io \abs{u_n\rchi_{\{u_n\leq0\}}}
    \leq
    0.
    $$
    Thus, integrating from $0$ to $t$, we have $u_n(t)\geq 0$ a.e. for every $t\in(0,T)$.
\end{proof}

We now prove that, for every $n\in\N$, the function $u_n$ is bounded. This will allow us to choose powers of $u_n$ as test functions in \eqref{eq:approx_KS_system} (without requiring a truncation). To prove the boundedness of $u_n$, we will use the following Lemma from \cite{StampacchiaElliptic}.

\begin{lemma}\label{LemmaStampacchiaLinfty}
    Let $\lambda: \mathbb{R}^{+} \rightarrow \mathbb{R}^{+}$ be a nonincreasing function such that
    $$
    \lambda(h) \leq \frac{M \lambda(k)^\delta}{(h-k)^\gamma}, \quad \forall\, h>k \geq 0,
    $$
    for $M>0, \delta>1$ and $\gamma>0$. Then $\lambda$ has a zero $d\in\R$, where
    $$
    d^\gamma=M \lambda(0)^{\delta-1} 2^{\frac{\delta \gamma}{\delta-1}} .
    $$
\end{lemma}

We will also use the following parabolic version of the Gagliardo-Nirenberg inequality, which follows from the interpolation inequality for Lebesgue spaces.

\begin{prop}\label{GNSparabTheo}
    Let $v\in\elle\infty(0,T,\elleom2)\cap\elle2(0,T;\Hom)$. Then $v\in\elle{2\frac{N+2}{N}}(\Omega_T)$  and
    \begin{equation}\label{GNSparabEQ}
        \ioT \abs{v}^{2\frac{N+2}{N}}
        \leq C
        \norma{v}{\elle\infty(0,T;\elleom2)}^\frac{4}{N}\norma{\nabla v}{\elle2(\Omega_T)}^2,      
    \end{equation}
    where $C$ is a positive constant which depends on $N$ and $\Omega$.
\end{prop}

\begin{lemma}\label{lemma:stima_un_infty:dep_n}
    For any $n\in\N$, $u_n$ belongs to $\elle\infty(\Omega_T)$.
\end{lemma}
\begin{proof}
    Let $\Theta_n(s)$ be the primitive of $T_n(s)$ and let $k\geq0$. We choose $\left[\Theta_n(u_n)-\Theta_n(k)\right]^+$ as a test function in the second equation of \eqref{eq:approx_KS_system}. This leads us to
    $$
    \io M(x)\nabla\psi_n\nabla\left[\Theta_n(u_n)-\Theta_n(k)\right]^+=\io \abs{T_n(u_n)}^\theta\left[\Theta_n(u_n)-\Theta_n(k)\right]^+.
    $$
    Notice that 
    \begin{equation}\label{GradTestu_nlimparab1}
        \nabla\left[\Theta_n(u_n)-\Theta_n(k)\right]^+=T_n(u_n)\nabla G_k(u_n)
    \end{equation}
    and that (recall that $u_n\geq0$)
    \begin{equation}\label{StimaTestu_nlimparab1}
        \left[\Theta_n(u_n)-\Theta_n(k)\right]^+=\int_k^{u_n}T_n(\sigma)\mathrm{d}\sigma\rchi_{\{u_n\geq k\}}\leq nG_k(u_n).
    \end{equation}    
    We now choose $G_k(u_n)$ as a test function in the first equation of \eqref{eq:approx_KS_system} to obtain (integrating from $0$ to $t$)
    $$
    \iot (u_n)_t G_k(u_n)+\alpha\iot\abs{\nabla G_k(u_n)}^2
    \leq 
    \iot T_n(u_n)M(x)\nabla\psi_n\nabla G_k(u_n)+\iot T_n(f) G_k(u_n),
    $$
    that is (here we observe that $(u_n)_t G_k(u_n)=\frac{d}{dt}\frac12 G_k(u_n)^2$ and apply \eqref{GradTestu_nlimparab1}),
    $$
    \frac12\io  G_k(u_n(t))^2+\alpha\iot\abs{\nabla G_k(u_n)}^2
    \leq 
    \iot \abs{T_n(u_n)}^\theta\left[\Theta_n(u_n)-\Theta_n(k)\right]^++\iot n G_k(u_n),
    $$
    which implies, thanks to \eqref{StimaTestu_nlimparab1}
    $$
    \frac12\io  G_k(u_n(t))^2+\alpha\iot\abs{\nabla G_k(u_n)}^2
    \leq 
    \ioT n^{\theta+1}G_k(u_n)+\ioT n G_k(u_n).
    $$
    We thus have, since both terms on the left hand side are positive, there exists a constant $C>0$ such that
    \begin{equation}
        \left\{\begin{aligned}
        & \norma{G_k(u_n)}{\elle\infty(0,T;\elleom2)}^2\leq C n^{\theta+1}\int_{\Omega_T}G_k(u_n)\\
        & \norma{\nabla G_k(u_n)}{\elleomT2}^2\leq C n^{\theta+1}\int_{\Omega_T}G_k(u_n).
        \end{aligned}\right.
    \end{equation}
    By Proposition \ref{GNSparabTheo}, it follows that there exists a positive constant $C(N,\Omega,n)$ such that
    \begin{equation}\label{eq:bound_un:costante_expl}
    \int_{\Omega_T}G_k(u_n)^{2\frac{N+2}{N}}\leq C(N,\Omega,n) \left(\int_{u_n\geq k}G_k(u_n)\right)^{\frac{N+2}{N}},
    \end{equation} for some $C(N,\Omega,n)>0$.
    We now define 
    \begin{equation}\label{proof:bounded:def_Ak}
        A_k=\{(x,t)\in\Omega_T\,\colon \,u_n(x,t)\geq k\}
    \end{equation}
    and apply the H\"older inequality to the last term, obtaining (up to increasing the constant $C(N,\Omega,n)$)
    \[
    \int_{A_k}G_k(u_n)^{2\frac{N+2}{N}}\leq C(N,\Omega,n)\left(\int_{A_k}G_k(u_n)^{2\frac{N+2}{N}}\right)^{\frac12}\abs{A_k}^\frac{N+4}{2N+4}.
    \]
    Now choose $h>k$ to obtain (after cancellations)
    $$
    (h-k)^\frac{N+2}{N}\abs{A_h}^\frac{1}{2}\leq C(N,\Omega,n) \abs{A_k}^\frac{N+4}{2N+4},
    $$
    which implies
    $$
    \abs{A_h}\leq C(N,\Omega,n) \frac{\abs{A_k}^\frac{2N+8}{2N+4}}{(h-k)^{2\frac{N+2}{N}}}.
    $$
    Observing that $\frac{2N+8}{2N+4}>1$, we now invoke Lemma \ref{LemmaStampacchiaLinfty} with $$\lambda(k)=\abs{A_k}$$ to conclude that $\abs{A_{k_n}}=0$ for some $k_n\geq0$, which ends the proof. 
\end{proof}

Observe that, since the constant $C(N,\Omega,n)$ appearing in \eqref{eq:bound_un:costante_expl} diverges as $n\to\infty$, the estimate used to prove Lemma \ref{lemma:stima_un_infty:dep_n} depends on $n$. In other words, we cannot derive a uniform bound on $(u_n)_n$ in $\elleomT\infty$ from this result. 

The following two results only require a control on the $\elleomT1$ norm of the datum $f$ and, thus, are independent of $n$. First, we prove an estimate on the $\elle\infty(0,T;\elleom1)$ norm of $u_n$ in the spirit of the $\elle1$ estimate proved in \cite{BoccardoJDE,BO-Kellersegel}.

\begin{lemma}\label{StimaL1ParabEll}
Let $u_n$ be a solution to the first equation to \eqref{eq:approx_KS_system}, then $u_n\in\elle\infty(0,T;\elleom1)$ and
$$\norma{u_n}{\elle\infty(0,T;\elleom1)}\leq \norma{f}{\elle1(\Omega_T)}.$$
\end{lemma}
\begin{proof}
    Choose $T_h(u_n)$ as a test function at a fixed time $t$ to obtain
    $$
    \langle (u_n)_t,T_h(u_n)\rangle+\alpha\io\abs{\nabla T_h(u_n)}^2\leq 
    \io T_n(u_n) M(x)\nabla\psi_n\cdot\nabla T_h(u_n) +\io f(x,t)T_h(u_n).
    $$
    Applying Young's inequality to the right hand side (and recalling that $\abs{T_h(u_n)}\leq h$) we obtain
    $$
    \langle (u_n)_t,T_h(u_n)\rangle+\alpha\io\abs{\nabla T_h(u_n)}^2\leq 
    \frac{\beta^2 h^2}{2\alpha} \io\abs{\nabla\psi_n}^2+\frac{\alpha}{2}\io\abs{\nabla T_h(u_n)}^2 +\io h\abs{f(x,t)}
    $$
    After simplifying, we divide by $h$ and pass to the limit as $h\to0$, obtaining
    $$
    \frac{d}{dt}\io u_n(t)\leq \io \abs{f(x,t)}.
    $$
    Integrating from $0$ to $t$ leads to
    $$\norma{u_n(t)}{\elle1}\leq \int_0^t\io \abs{f(x,t)}\leq \ioT\abs{f}= \norma{f}{\elle1(\Omega_T)}$$
    which, taking the supremum over all $t\in(0,T)$, concludes the proof.
\end{proof}

We now focus on the solution $\psi_n$. In the next Lemma, we show that, for any datum $f\in\elleomT1$, uniform bounds on the sequence $(\psi_n)_n$ can be obtained.
    
\begin{lemma}\label{lemma:psin_bounded_unif}
    The sequence $(\psi_n)_n$ is bounded in $\elle\infty(\Omega_T)$ and $(\nabla\psi_n)_n$ is bounded in $\elle\infty(0,T;\elleom2)$. Moreover, there exists a constant $C>0$ such that the following estimate holds:
    \begin{equation}\label{eq:lemma:psin_bounded_unif}
        \norma{\psi_n}{\elle\infty(\Omega_T)}+\norma{\nabla \psi_n}{\elle\infty(0,T;\elleom2)}\leq C\norma{f}{\elle1(\Omega_T)}^\theta.
    \end{equation}
\end{lemma}
\begin{proof}
    First, observe that $T_n(u_n)^\theta\leq u_n^\theta$ and that, by the H\"older inequality and Lemma \ref{StimaL1ParabEll}, there exists a constant $C>0$ such that
    $$\io u_n(t)^\theta\leq C\left(\io u_n(t)\right)^\theta\leq C\norma{f}{\elle1(\Omega_T)}^\theta \quad\forall\,t\in(0,T).$$
    Recalling that $\frac1\theta>\frac{N}{2}$, we apply \cite[Th\'eor\`eme 4.1]{StampacchiaElliptic} to obtain
    \begin{equation}\label{eq:proof:lemmapsinbounded_1}
        \norma{\psi_n(t)}{\elleom\infty}\leq C\norma{u_n^\theta}{\elleom{\frac{1}{\theta}}}\leq C\norma{f}{\elle1(\Omega_T)}^\theta \quad\forall\,t\in(0,T).
    \end{equation}
    To estimate $\nabla\psi_n$, we choose $\psi_n$ as a test function in the second equation of \eqref{eq:approx_KS_system}, obtaining
    $$
    \alpha\io\abs{\nabla\psi_n(t)}^2
    \leq\io T_n(u_n)^\theta\psi_n
    \leq
    \norma{u_n}{\elleom1}^\theta\norma{\psi_n(t)}{\elleom\infty}\leq \bigr(C\norma{f}{\elle1(\Omega_T)}^\theta)^2\quad\forall\,t\in(0,T),
    $$
    from which we have
    \begin{equation}\label{eq:proof:lemmapsinbounded_2}
    \norma{\nabla\psi_n(t)}{\elleom2}\leq \frac{1}{\alpha^{\frac12}}C\norma{f}{\elle1(\Omega_T)}^\theta\quad\forall\,t\in(0,T).        
    \end{equation}
    Taking the supremum over $t$ in \eqref{eq:proof:lemmapsinbounded_1} and \eqref{eq:proof:lemmapsinbounded_2} (up to increasing the constant $C$) concludes the proof.
\end{proof}

\subsection{Existence and regularity of a solution with finite energy}\label{sec:highsummability}

In this section, we focus on the finite energy case: that is, when $f\in\elleomT m$ for some $m\geq \frac{2N+4}{N+4}$. Depending on the summability of $f$, we prove two a priori estimates. 
Lemma \ref{lemma:stima_m**_unif} shows that, if $f$ belongs to $\elle m (\Omega_T)$ for some $\frac{2N+4}{N+4}\leq m<\frac{N+2}{2}$, the sequence  $(u_n)_n$ is bounded in $\elleomT {m^{\star\star}}$ and in $\elle2(0,T;\Hom)$. If the summability of $f$ is higher (that is, if $f\in\elleomT m$ for $m>\frac{N+2}{2}$), Lemma \ref{lemma:stima_Linfty_unif} shows that the sequence $(u_n)_n$ is bounded in $\elleomT \infty$.

\begin{remark}
    Since $\Omega_T$ has finite measure, $f\in\elleomT p$ for some $p>1$ implies $f\in \elleomT q$ for every $q\in [1,p]$. Thus, the assumptions of Lemma \ref{lemma:stima_m**_unif} are satisfied throughout the whole section.
\end{remark}

\begin{lemma}\label{lemma:stima_m**_unif}
    Let $f\in\elle m (\Omega_T)$ with $\frac{2N+4}{N+4}\leq m<\frac{N+2}{2}$, then the sequence of approximating solutions $(u_n)_n$ is bounded in $\elleomT {m^{\star\star}}$ and in $\elle2(0,T;\Hom)$.
\end{lemma}
\begin{proof}
    Let $\gamma\geq1$. We define the function
    $$F_n(s)=\int_0^s \sigma^{2\gamma-2}T_n(\sigma) \mathrm{d}\sigma$$
    and choose $F_n(u_n)$ as a test function in the second equation of \eqref{eq:approx_KS_system} to get
    $$
    \io u_n^{2\gamma-2}T_n(u_n) M(x)\nabla \psi_n\cdot\nabla u_n 
    =
    \io M(x)\nabla \psi_n\cdot\nabla F_n(u_n)
    =
    \io T_n(u_n)^\theta F_n(u_n)
    $$
    $$
    \leq\frac{1}{2\gamma}\io T_n(u_n)^\theta u_n^{2\gamma}
    $$
    Now choose $u_n^{2\gamma-1}$ (which belongs to $\elleomT\infty$ by Lemma \ref{lemma:stima_un_infty:dep_n}) as a test function in the first equation of \eqref{eq:approx_KS_system} to get
    \begin{equation*}
        \begin{split}
            \frac{1}{{2\gamma}}\frac{d}{dt}\io u_n^{2\gamma}+&\alpha(2\gamma-1)\io\abs{\nabla u_n}^2u_n^{2\gamma-2}\\&
    \leq
    (2\gamma-1)\io u_n^{2\gamma-2}T_n(u_n) M(x)\nabla \psi_n\cdot\nabla u_n+\io f(x,t)u_n^{2\gamma-1}.
        \end{split}
    \end{equation*}
    Thus
    \begin{equation}\label{basta-gamma-geq-frac12}
        \frac{1}{{2\gamma}}\frac{d}{dt}\io u_n^{2\gamma}
        +\frac{\alpha(2\gamma-1)}{\gamma^2}\io\abs{\nabla u_n^\gamma}^2
        \leq
        \frac{2\gamma-1}{2\gamma}\io u_n^\theta u_n^{2\gamma}+\io f(x,t)u_n^{2\gamma-1}.
    \end{equation}
    Now we split the integration domain in the third term and apply the H\"older inequality to get
\begin{equation}
        \begin{split}
        \io u_n^{\theta+2\gamma}=&\int_{u_n\leq k} u_n^{\theta+2\gamma}+\int_{u_n\geq k} u_n^{\theta+2\gamma}\\& \leq 
        k^{\theta+2\gamma}\abs{\Omega}+
        \left(\int_{u_n\geq k}u_n^{\theta\frac{N}{2}}\right)^{\frac{2}{N}}\left(\int_{u_n\geq k}u_n^{\gamma {\frac{2N}{N-2}}}\right)^\frac{N-2}{N}
        \\&
        \leq
        k^{\theta+2\gamma}\abs{\Omega}+
        \mathcal{S}\left(\frac{\norma{f}{\elle1(\Omega_T)}}{k^{1-\theta\frac{N}{2}}}\right)^{\frac{2}{N}}\io\abs{\nabla u_n^\gamma}^2        
    \end{split}
\end{equation}
where the last estimate comes from the Sobolev inequality.
    Thus, choosing $k$ large enough, we have
    $$\frac{1}{{2\gamma}}\frac{d}{dt}\io u_n^{2\gamma}+
    \frac{\alpha(2\gamma-1)}{2\gamma^2}\io\abs{\nabla u_n^\gamma}^2
    \leq
     (2\gamma-1)k^{\theta+2\gamma}\abs{\Omega}+\io f(x,t)u_n^{2\gamma-1}.$$
    Integration from $0$ to $t$ leads to
    $$\frac{1}{{2\gamma}}\io u_n^{2\gamma}(t)+
    \frac{\alpha(2\gamma-1)}{2\gamma^2}\iot\abs{\nabla u_n^\gamma}^2
    \leq
     (2\gamma-1)k^{\theta+2\gamma}\abs{\Omega}T+\ioT f(x,t)u_n^{2\gamma-1}.$$
    Since both terms on the left hand side are positive, it follows that there exists a constant $C>0$ (which we will tacitly increase if needed) such that
    \begin{equation}\label{StimaGNSParabEllStampacchia2}
        \left\{\begin{aligned}
        & \norma{u_n^\gamma}{\elle\infty(0,T;\elleom2)}^2\leq C\left(k^{\theta+2\gamma}\abs{\Omega}T+\ioT f(x,t)u_n^{2\gamma-1}\right).\\
        & \norma{\nabla u_n^\gamma}{\elleomT2}^2\leq C\left(k^{\theta+2\gamma}\abs{\Omega}T+\ioT f(x,t)u_n^{2\gamma-1}\right).
        \end{aligned}\right.
    \end{equation}
    Hence, applying Proposition \ref{GNSparabTheo} to $v=u_n^\gamma$, we obtain
    $$
    \ioT u_n^{2\gamma\frac{N+2}{N}}
    %\leq\norma{u_n^\gamma}{\elle\infty(0,T;\elleom2)}^\frac{4}{N}\norma{\nabla u_n^\gamma}{\elleomT2}^2
    \leq C\left(k^{\theta+2\gamma}\abs{\Omega}T+\ioT f(x,t)u_n^{2\gamma-1}\right)^\frac{N+2}{N}.
    $$
    We now apply the H\"older inequality to the last term, obtaining
    \begin{equation}\label{HolderRHSParabellstampacchia}
        \ioT f(x,t)u_n^{2\gamma-1}
        \leq
        \norma{f}{\elleomT m}\left(\ioT u_n^{(2\gamma-1)m'}\right)^\frac{1}{m'}.
    \end{equation}
    Setting $2\gamma\frac{N+2}{N}=(2\gamma-1)m'$ leads to
    $$\gamma=\frac{Nm}{2(N+2-2m)}\quad\text{and}\quad 2\gamma\frac{N+2}{N}= \frac{(N+2)m}{N+2-2m}=m^{\star\star},$$
    hence
    \begin{equation}\label{proof:stima_m**_1:slower}        
    \left(\ioT u_n^{\frac{m(N+2)}{N+2-2m}}\right)^\frac{N}{N+2}
    \leq
    C\left(k^{\theta+2\gamma}\abs{\Omega}T+
    \norma{f}{\elleomT m}\left(\ioT u_n^\frac{m(N+2)}{N+2-2m}\right)^\frac{1}{m'}\right).
    \end{equation}
    Moreover, we have
    $$\frac{1}{m'}=1-\frac{1}{m}<1-\frac{2}{N+2}=\frac{N}{N+2}.$$
    This implies that the right-hand side of \eqref{proof:stima_m**_1:slower}, seen as a function of $\ioT u_n^\frac{m(N+2)}{N+2-2m}$, grows at a slower rate than the term on the left-hand side. 
    It follows that the sequence $(u_n)_n$ is bounded in $\elleomT {m^{\star\star}}$.

    Observe that, since $f\in\elleomT{\frac{2N+4}{N+4}}$, we can repeat the same calculations with
    $m=\frac{2N+4}{N+4}$ and $\gamma=1$, which implies, thanks to \eqref{StimaGNSParabEllStampacchia2}, that the sequence $(u_n)_n$ is bounded in $\elle2(0,T;\Hom)$.
\end{proof}

Note that, although
\begin{equation}\label{gamma1-parabell}
    \gamma\geq1\qquad\Longleftrightarrow\qquad m\geq\frac{2N+4}{N+4},
\end{equation}
the requirement $\gamma\geq1$ in the previous proof could be relaxed to $\gamma>\frac{1}{2}$ after \eqref{basta-gamma-geq-frac12}. Indeed, a similar approach (with a slight modification of the test functions) will be used to prove Lemma \ref{lemma:stima_m**_unif:infinite_energy}, which deals with the case $1<m\leq\frac{2N+4}{N+4}$. 

\begin{remark}
    Thanks to \eqref{StimaGNSParabEllStampacchia2} and \eqref{HolderRHSParabellstampacchia}, the sequence $(u_n)_n$ is bounded in $\elle\infty(0,T;\elleom {2\gamma})$, which is $\elle\infty(0,T;\elleom{\frac{Nm}{N+2-2m}})$. 
\end{remark}

\begin{lemma}\label{lemma:stima_Linfty_unif}
    Let $m>\frac{N+2}{2}$ and $f\in\elle m (\Omega_T)$. Then, the sequence of approximating solutions $(u_n)_n$ is bounded in $\elleomT {\infty}$.
\end{lemma}

\begin{proof}
   Choosing the same test functions we chose in the proof of Lemma \ref{lemma:stima_un_infty:dep_n}, we get
    $$
    \frac12 \io G_k(u_n(t))^2+\alpha\iot\abs{\nabla G_k(u_n)}^2
    \leq 
    \iot \abs{T_n(u_n)}^\theta\left[\Theta_n(u_n)-\Theta_n(k)\right]^++\iot T_n(f) G_k(u_n).
    $$
    Observe that
    \begin{equation}
        \left[\Theta_n(u_n)-\Theta_n(k)\right]^+=\int_k^{u_n}T_n(\sigma)\mathrm{d}\sigma\rchi_{\{u_n\geq k\}}\leq u_n G_k(u_n),
    \end{equation}
    which implies
    $$
    \frac12 \io G_k(u_n(t))^2+\alpha\iot\abs{\nabla G_k(u_n)}^2
    \leq 
    \ioT u_n^{\theta+1}G_k(u_n)+\ioT f(x,t) G_k(u_n).
    $$
    Define the function $\tilde f_n(x,t)=u_n^{\theta+1}(x,t)+f(x,t)$, so that the right-hand side can be written as follows:
    \begin{equation}
        \frac12 \io G_k(u_n(t))^2+\alpha\iot\abs{\nabla G_k(u_n)}^2
        \leq 
        \ioT \tilde f_n(x,t) G_k(u_n).
    \end{equation}

    Observe that, by Lemma \ref{lemma:stima_m**_unif}, the sequence $(u_n)_n$ is bounded in $\elleomT p$ for all $p<\infty$, which implies that the sequence $(\tilde f_n)_n$ is bounded in $\elleomT m$. In particular, there exists a constant $C>0$ (which we will tacitly increase if needed) such that
    \begin{equation}\label{Bound_fntilde}
        \norma{\tilde f_n}{\elleomT m}\leq  C(1+ \norma{f}{\elleomT m}).
    \end{equation}

    By repeating the argument used in the proof of Lemma \ref{lemma:stima_m**_unif}, we obtain
    \begin{equation}
        \left\{\begin{aligned}
        & \norma{G_k(u_n)}{\elle\infty(0,T;\elleom2)}^2\leq C \int_{\Omega_T}\tilde f_n(x,t)G_k(u_n)\\
        & \norma{\nabla G_k(u_n)}{\elleomT2}^2\leq C \int_{\Omega_T}\tilde f_n(x,t)G_k(u_n).
        \end{aligned}\right.
    \end{equation}
    Which, by Proposition \ref{GNSparabTheo}, implies 
    \begin{equation}\label{proof:lemma:stimaLinfty_unif:eq1}
    \int_{\Omega_T}G_k(u_n)^{2\frac{N+2}{N}}
    %\leq \norma{G_k(u_n)}{\elle\infty(0,T;\elleom2)}^\frac{4}{N}\norma{\nabla G_k(u_n)}{\elleomT2}^2
    \leq 
    C \left(\int_{u_n\geq k}\tilde f_n(x,t)G_k(u_n)\right)^{\frac{N+2}{N}}.
    \end{equation}
    We now proceed as in the proof of Lemma \ref{lemma:stima_un_infty:dep_n}: consider the set $A_k$ defined in \eqref{proof:bounded:def_Ak} and apply the H\"older inequality with exponent $2\frac{N+2}{N}$ to the last term of \eqref{proof:lemma:stimaLinfty_unif:eq1}, obtaining
    $$
    \int_{A_k}G_k(u_n)^{2\frac{N+2}{N}}\leq C\left(\int_{A_k}G_k(u_n)^{2\frac{N+2}{N}}\right)^{\frac12}\left(\int_{A_k}\tilde f_n(x,t)^{\frac{2N+4}{N+4}}\right)^{\frac{N+4}{2N}},
    $$
    which implies
    $$
    \int_{A_k}G_k(u_n)^{2\frac{N+2}{N}}\leq C\left(\int_{A_k}\tilde f_n(x,t)^{\frac{2N+4}{N+4}}\right)^{\frac{N+4}{N}}.
    $$
    Since $m>\frac{N+2}{2}>\frac{2 N+4}{N+4}$, a further use of the H\"older inequality, along with \eqref{Bound_fntilde}, yields
    $$
    \int_{A_k} G_k(u)^{2 \frac{N+2}{N}} \leq C(1+\norma{f}{\elle m(\Omega_T)})^{\frac{2(N+2)}{N}} \abs{A_k}^{\frac{N+4}{N}-\frac{2(N+2)}{N m}} .
    $$
    Let $h>k$. After cancellations, we obtain
    $$
    \abs{A_h}(h-k)^{2 \frac{N+2}{N}} \leq \int_{A_k} G_k(u)^{2 \frac{N+2}{N}}
    \leq C  (1+\norma{f}{\elle m(\Omega_T)})^{\frac{2(N+2)}{N}}
    \abs{A_k}^{\frac{N+4}{N}-\frac{2(N+2)}{Nm}}
    $$
    Setting $\lambda(h)=\abs{A_h}$ and recalling that $m>\frac{N+2}{2}$, we can apply Lemma \ref{LemmaStampacchiaLinfty} to conclude.
    % that there exists a constant $d$, depending only on $m$, $\norma{f}{\elle m(\Omega_T)}$ and $\alpha$, such that $\lambda(d)=0$, that is
    % $$
    % \norma{u_n}{\elle \infty(\Omega_T)} \leq d,
    % $$
    % which concludes the proof.
\end{proof}

We can now prove the existence of a solution in the finite-energy case. Using the bounds established so far, we pass to the limit in \eqref{eq:approx_KS_system} as $n\to+\infty$. In this way, we obtain a solution $(u,\psi)$ to \eqref{eq:KS-Sys-ParabEll} with the summability properties stated in Lemma \ref{lemma:stima_Linfty_unif} and Lemma \ref{lemma:stima_m**_unif}.

Although the proofs of Theorem \ref{Main_Theorem_1} and Theorem \ref{Main_Theorem_2} are based on the same general strategy, some technical details differ slightly. For clarity, we therefore present them separately.

\begin{proof}[Proof of Theorem \ref{Main_Theorem_1}]
    We write the first equation of \eqref{eq:approx_KS_system} as
    $$(u_n)_t=\operatorname{div}(A(x, t) \nabla u_n)-\operatorname{div}(T_n(u_n) M(x) \nabla \psi_n)+T_n(f(x,t))$$
    and observe that, by Lemma \ref{lemma:stima_m**_unif} and Lemma \ref{lemma:stima_Linfty_unif},
    \begin{itemize}
        \item $\operatorname{div}(A(x, t) \nabla u_n)$ is bounded in $\elle2(0,T;H^{-1}(\Omega))$;
        \item $T_n(u_n) M(x) \nabla \psi_n$ is bounded in $\elleomT2$;
        \item  $T_n(f(x,t))$ is bounded in $\elleomT m$.
    \end{itemize}
    Thus, the sequence $((u_n)_t)_n$ in bounded in $\elle2(0,T;H^{-1}(\Omega))$.
    We now apply \cite[Corollary 4]{Simon} to obtain that the sequence $(u_n)_n$ has a converging subsequence in $\elleomT2$. Let $u$ be the limit of such a subsequence: we have that $u\in\elle2(0,T;\Hom)$ (since $(u_n)_n$ is bounded in $\elle2(0,T;\Hom)$ and $\nabla u_n\rightharpoonup\nabla u$). Moreover, $u_t\in\elle2(0,T;H^{-1}(\Omega))$ and, by Lemma \ref{lemma:stima_Linfty_unif}, $u\in\elleomT\infty$. This also implies (see \cite[Section 5.9.2]{EvansPDE}) that $u\in C^0([0,T];\elleom2)$.
    
    By Lemma \ref{lemma:psin_bounded_unif}, there exists a subsequence of $(\psi_n)_n$ that strongly converges in $\elleomT2$ and weakly converges in $\elle2(0,T;\Hom)$ to some function $\psi(x,t)$ which is a weak solution to
    \[
    -\operatorname{div}(M(x)\nabla\psi)=u^\theta(x,t)\qquad \text{in }\Omega
    \]
    for every $t\in (0,T)$, with zero Dirichlet boundary conditions.
    Moreover, since $u_n^\theta$ is strongly convergent in $\elleomT p$ with $p>\frac{N}{2}$, by applying standard elliptic estimates \cite{BoccardoCroce} to the equation
    \[
    -\operatorname{div}(M(x)\nabla(\psi_n-\psi))=u_n^\theta-u^\theta,
    \]
    we get (up to subsequences):
    \begin{itemize}
        \item $\psi_n\to\psi$ in $\elle\infty(0,T;\Hom)$, which implies that $\psi\in C^0([0,T];\Hom)$; 
        \item $\psi_n\to\psi$ in $\elle\infty(\Omega_T)$, which implies that $\psi\in \elle\infty(\Omega_T)$.
    \end{itemize}

    Let $\varphi$ and $v$ be test functions as in the statement of Theorem \ref{Main_Theorem_1}. Then, we take the limit (up to subsequences) as $n\to\infty$ in 
    \begin{equation}
        \begin{dcases}
                \int_0^T \langle (u_n)_t,\varphi\rangle+\ioT A(x,t)\nabla u_n\cdot \nabla \varphi = \ioT T_n(u_n) M(x)\nabla \psi_n\cdot\nabla\varphi + \ioT T_n(f)\varphi\\
                \io M(x)\nabla \psi_n(t)\cdot\nabla v = \io u_n(t)^\theta v,
        \end{dcases}
    \end{equation}
    to conclude that the couple $(u,\psi)$ is a solution to \eqref{eq:KS-Sys-ParabEll} in the sense \eqref{DefSol_parabell_theo1}.
\end{proof}

\begin{proof}[Proof of Theorem \ref{Main_Theorem_2}]
    We write the first equation of \eqref{eq:approx_KS_system} as
    $$(u_n)_t=\operatorname{div}(A(x, t) \nabla u_n)-\operatorname{div}(T_n(u_n) M(x) \nabla \psi_n)+T_n(f(x,t))$$
    and observe that, by Lemma \ref{lemma:stima_m**_unif}:
    \begin{itemize}
        \item $\operatorname{div}(A(x, t) \nabla u_n)$ is bounded in $\elle2(0,T;H^{-1}(\Omega))$;
        \item $T_n(u_n) M(x) \nabla \psi_n$ is bounded in $\elleomT p$ with $\frac1p=\frac12+\frac 1m-\frac2{N+2}$ (note that $p\geq1$)
        \item  $T_n(f(x,t))$ is bounded in $\elleomT m$.
    \end{itemize}
    This implies that the sequence $((u_n)_t)_n$ in bounded in $\elle1(0,T;W^{-1,1}(\Omega))$.
    Thus, we can apply \cite[Corollary 4]{Simon} to obtain that the sequence $(u_n)_n$ has a converging subsequence in $\elleomT2$. Let $u$ be the limit of such a subsequence: we have that $u\in\elle2(0,T;\Hom)$ (since $(u_n)_n$ is bounded in $\elle2(0,T;\Hom)$ and since $\nabla u_n\rightharpoonup\nabla u$).

    By Lemma \ref{lemma:psin_bounded_unif}, there exists a subsequence of $(\psi_n)_n$ that strongly converges in $\elleomT2$ and weakly converges in $\elle2(0,T;\Hom)$ to some function $\psi(x,t)$ which is a weak solution to
    \[
    -\operatorname{div}(M(x)\nabla\psi)=u^\theta(x,t)\qquad \text{in }\Omega
    \]
    for every $t\in (0,T)$, with zero Dirichlet boundary conditions.
    Moreover, since $u_n^\theta$ is strongly convergent in $\elleomT p$ with $p>\frac{N}{2}$, by applying standard elliptic estimates \cite{BoccardoCroce} to the equation
    \[
    -\operatorname{div}(M(x)\nabla(\psi_n-\psi))=u_n^\theta-u^\theta,
    \]
    we get (up to subsequences):
    \begin{itemize}
        \item $\psi_n\to\psi$ in $\elle\infty(0,T;\Hom)$, which implies that $\psi\in C^0([0,T];\Hom)$; 
        \item $\psi_n\to\psi$ in $\elle\infty(\Omega_T)$, which implies that $\psi\in \elle\infty(\Omega_T)$.
    \end{itemize}

    Let $\varphi$ and $v$ be test functions as in the statement of Theorem \ref{Main_Theorem_2}. Then, we take the limit (up to subsequences) as $n\to\infty$ in 
    \begin{equation}
        \begin{dcases}
                \int_0^T \langle (u_n)_t,\varphi\rangle+\ioT A(x,t)\nabla u_n\cdot \nabla \varphi = \ioT T_n(u_n) M(x)\nabla \psi_n\cdot\nabla\varphi + \ioT T_n(f)\varphi\\
                \io M(x)\nabla \psi_n(t)\cdot\nabla v = \io u_n(t)^\theta v,
        \end{dcases}
    \end{equation}
    to conclude that the couple $(u,\psi)$ is a solution to \eqref{eq:KS-Sys-ParabEll} in the sense \eqref{DefSol_parabell_theo2}.
\end{proof}

\begin{remark}
    Recall that, due to \eqref{StimaGNSParabEllStampacchia2} and \eqref{HolderRHSParabellstampacchia}, the sequence $(u_n)_n$ is also bounded in $\elle\infty(0,T;\elleom {2\gamma})$ and in $\elleomT{2\gamma\frac{N+2}{N}}$ with $2\gamma=\frac{Nm}{N+2-2m}$. We can interpolate between the two spaces with exponents $r$ and $q$ such that (letting $\lambda\in[0,1]$)
    $$\frac{1}{r}=\frac{\lambda}{2\gamma\frac{N+2}{N}}\quad\quad \frac{1}{q}=\frac{1-\lambda}{2\gamma}+\frac{\lambda}{2\gamma\frac{N+2}{N}}$$
    to obtain
    $$\int_0^T \norma{u_n}{\elleom q}^r
    \leq 
    \int_0^T\norma{u_n}{\elle\infty(\elleom {2\gamma})}^{r(1-\lambda)}\norma{u_n}{2\gamma\frac{N+2}{N}}^{r\lambda}
    =\norma{u_n}{\elle\infty(\elleom {2\gamma})}^{r(1-\lambda)}\norma {u_n}{\elleomT{2\gamma\frac{N+2}{N}}}.$$
    It follows that $(u_n)_n$ is bounded in $\elle r(0,T;\elleom q)$ and thus, by Fatou's Lemma, $u\in \elle r(0,T;\elleom q)$.
\end{remark}

\subsection{Existence of a distributional solution with singular data}\label{sec:distrib_sol}

If the datum $f$ belongs to $\elleomT m$ with $m\in[1,\frac{2N+4}{N+4})$, we can prove the existence of a solution found by approximation. Recall that all the properties of the sequence $(\psi_n)_n$ proved in the previous sections only require $f$ to belong to $\elleomT1$.

We now prove a boundedness result for the sequence $(u_n)_n$ which extends Lemma \ref{lemma:stima_m**_unif} to the range $1<m<\frac{2N+4}{N+4}$.
\begin{lemma}\label{lemma:stima_m**_unif:infinite_energy}
    Let $f\in\elle m (\Omega_T)$ with $1<m<\frac{2N+4}{N+4}$, then the sequence $(u_n)_n$ is bounded in $\elleomT {m^{\star\star}}$.
\end{lemma}
\begin{proof}
    Let $\gamma>\frac{1}{2}$ and $\varepsilon>0$. We define the function
    $$F_n^\varepsilon(s)=\int_0^s (\sigma+\varepsilon)^{2\gamma-2}T_n(\sigma) \mathrm{d}\sigma$$
    and choose $F_n^\varepsilon(u_n)$ as a test function in the second equation of \eqref{eq:approx_KS_system} to get
    $$
    \io (u_n+\varepsilon)^{2\gamma-2}T_n(u_n) M(x)\nabla \psi_n\cdot\nabla u_n 
    =
    \io M(x)\nabla \psi_n\cdot\nabla F_n^\varepsilon(u_n)
    =
    \io T_n(u_n)^\theta F_n^\varepsilon(u_n)
    $$
    $$
    \leq \frac{1}{2\gamma}\io T_n(u_n)^\theta (u_n+\varepsilon)^{2\gamma}
    \leq
    \frac{1}{2\gamma}\io u_n^{\theta} (u_n+\varepsilon)^{2\gamma}
    $$
    where we used 
    $$F_n^\varepsilon(s)\leq \int_0^s (\sigma+\varepsilon)^{2\gamma-1}\mathrm{d}\sigma\leq (s+\varepsilon)^{2\gamma}-\varepsilon^{2\gamma}\leq (s+\varepsilon)^{2\gamma}.$$
    Now choose 
    $$g_\varepsilon(u_n)=\frac{(u_n+\varepsilon)^{2\gamma-1}-\varepsilon^{2\gamma-1}}{2\gamma-1}$$
    (which belongs to $\elleomT\infty$ by Lemma \ref{lemma:stima_un_infty:dep_n}) as a test function in the first equation of \eqref{eq:approx_KS_system} to get
    $$\langle (u_n)_t,g_\varepsilon(u_n)\rangle
    +\io(u_n+\varepsilon)^{2\gamma-2}A(x,t)\nabla u_n \cdot \nabla u_n$$
    $$
    \leq\io (u_n+\varepsilon)^{2\gamma-2}T_n(u_n) M(x)\nabla \psi_n\cdot\nabla u_n 
    +\io f(x,t)g_\varepsilon(u_n).$$
    Thus
    $$\langle (u_n)_t,g_\varepsilon(u_n)\rangle
    +\io(u_n+\varepsilon)^{2\gamma-2}A(x,t)\nabla u_n \cdot \nabla u_n
    \leq
    \frac{1}{2\gamma}\io u_n^{\theta} (u_n+\varepsilon)^{2\gamma}+\io f(x,t)g_\varepsilon(u_n).$$
    Integration from $0$ to $t$ leads to
    $$
    \io G_\varepsilon(u_n(t)) + \iot (u_n+\varepsilon)^{2\gamma-2}A(x,s)\nabla u_n \cdot \nabla u_n
    \leq
    \frac{1}{2\gamma}\iot u_n^{\theta} (u_n+\varepsilon)^{2\gamma}+\iot f(x,s)g_\varepsilon(u_n),
    $$
    where $G_\varepsilon$ is the primitive of $g_\varepsilon$ such that $G_\varepsilon(0)=0$.
    
    We can now apply Fatou's Lemma (notice that $G_\varepsilon(s)\to \frac{1}{2\gamma(2\gamma-1)}s^{2\gamma}$ as $\varepsilon\to0$ and that $G_\varepsilon(s)\geq 0$ for $s\geq 0$) to obtain
    $$\frac{1}{{2\gamma(2\gamma-1)}}\io u_n^{2\gamma}(t)
    +\frac{\alpha}{\gamma^2}\iot\abs{\nabla u_n^\gamma}^2
    \leq
    \frac{1}{2\gamma}\iot u_n^\theta u_n^{2\gamma}+\frac{1}{2\gamma-1}\iot f(x,s)u_n^{2\gamma-1}.$$
    At this point, it suffices to follow the proof of Lemma \ref{lemma:stima_m**_unif} to conclude.
\end{proof}
\begin{remark}\label{Stima-u_ngamma-bad}
    The previous result also implies that
    \begin{equation}\label{Remark-bound-u_ngamma}
        \left\{\begin{aligned}
        & \norma{u_n^\gamma}{\elle\infty(0,T;\elleom2)}^2\leq C\left(k^{\theta+2\gamma}\abs{\Omega}T+\ioT f(x,t)u_n^{2\gamma-1}\right).\\
        & \norma{\nabla u_n^\gamma}{\elleomT2}^2\leq C\left(k^{\theta+2\gamma}\abs{\Omega}T+\ioT f(x,t)u_n^{2\gamma-1}\right).
        \end{aligned}\right.
    \end{equation}
    Since $(2\gamma-1)m'=m^{\star\star}$, \eqref{Remark-bound-u_ngamma} implies that the sequence $(u_n^\gamma)_n$ is bounded in $\elle\infty(0,T;\elleom2)$ and in $\elle2(0,T;\Hom)$.
\end{remark}

Recall that, in the elliptic case (see e.g. \cite{BO-Kellersegel}), singular data (that is, $f\in\elleomT{m}$ with $m<\frac{2N}{N+2}$) prevent the solution from having finite Dirichlet energy. However, the sequence of approximating solutions is bounded in $\Sobom{\frac{Nm}{N-m}}$ whenever $1<m<\frac{2N}{N+2}$. 
Since the parabolic $L^p$-regularity results \cite{Aronson-Serrin,BDGO-parabolic} typically behave as the elliptic ones in dimension $N+2$,
we expect $(\nabla u_n)_n$ to be bounded in $\left(\elleomT{m^{\star}}\right)^N$. Indeed, this is the case.
\begin{lemma}\label{lemma:u_nW1m*_KS_parabell}
    Let $f\in\elleomT m$, with $1<m<\frac{2N+4}{N+4}$. Then the sequence of approximating solutions is bounded in $\elle {m^\star}(0,T;\Sobom{m^\star})$ (where ${m^\star}=\frac{(N+2)m}{(N+2)-m})$.
\end{lemma}
\begin{proof}
    To prove this result, we use the same technique employed in \cite{BoccardoJDE}. Our objective is to exploit the H\"older inequality, Lemma \ref{lemma:stima_m**_unif:infinite_energy} and the fact that $\frac12<\gamma<1$ (see \eqref{gamma1-parabell}) to gain an estimate on the $\elle {m^\star}$ norm of $\nabla u_n$.
    More precisely, we have

    $$\ioT \abs{\nabla u_n}^{m^\star} =
    \ioT \frac{\abs{\nabla u_n}^{m^\star}}{u_n^{{m^\star}(1-\gamma)}}u_n^{{m^\star}(1-\gamma)} 
    \leq
    \left(\ioT \frac{\abs{\nabla u_n}^2}{u_n^{2(1-\gamma)}}\right)^\frac{{m^\star}}{2}
    \left(\ioT u_n^{\frac{2{m^\star}}{2-{m^\star}}(1-\gamma)}\right)^{1-\frac{{m^\star}}{2}}.$$

    Now observe that, choosing $\gamma$ as in the proof of Lemma \ref{lemma:stima_m**_unif:infinite_energy}, we get
    $$\frac{2{m^\star}}{2-{m^\star}}(1-\gamma)=
    \frac{(N+2)m}{2N+4-4m-Nm}\frac{2N+4-4m-Nm}{(N+2)-2m}
    = \frac{(N+2)m}{(N+2)-2m}
    =m^{\star\star},$$
    which implies that, by Lemma \ref{lemma:stima_m**_unif:infinite_energy},
    $$\ioT u_n^{\frac{2{m^\star}}{2-{m^\star}}(1-\gamma)}=\ioT u_n^{m^{\star\star}}\leq C.$$
    On the other hand, notice that 
    $$\ioT \frac{\abs{\nabla u_n}^2}{u_n^{2(1-\gamma)}}=\frac{1}{\gamma^2}\norma{\nabla u_n^\gamma}{\elleomT2}^2,$$
    which is bounded by Remark \ref{Stima-u_ngamma-bad}. This concludes the proof.
\end{proof}

We now prove Theorem \ref{Main_Theorem_3}. Recall that, to apply \cite[Corollary 4]{Simon} to gain compactness of the sequence $(u_n)_n$ in some strong topology, we need a uniform bound on the time derivatives of the sequence $(u_n)_n$, which requires the convection term
$$-\operatorname{div}(T_n(u_n) M(x) \nabla \psi_n)$$ 
to be bounded in (at least) $\elle1(0,T;W^{-1,1}(\Omega))$. 
For this to be possible, since we only know that $(\nabla \psi_n)_n$ is bounded in $\left(\elleomT2\right)^N$, it is necessary for $(u_n)_n$ to be bounded in (at least) $\elleomT2$. Such requirement is equivalent to
$$m^{\star\star}=\frac{(N+2)m}{N+2-2m}\geq2,$$
which means that $m$ shall be greater than $\frac{2N+4}{N+6}$.

\begin{proof}[Proof of Theorem \ref{Main_Theorem_3}]
    We write the first equation of \eqref{eq:approx_KS_system} as
    $$(u_n)_t=\operatorname{div}(A(x, t) \nabla u_n)-\operatorname{div}(T_n(u_n) M(x) \nabla \psi_n)+T_n(f(x,t))$$
    and observe that, by Lemma \ref{lemma:u_nW1m*_KS_parabell}:
    \begin{itemize}
        \item $\operatorname{div}(A(x, t) \nabla u_n)$ is bounded in $\elle{({m^\star})'}(0,T;W^{-1,1}(\Omega))$;
        \item $T_n(u_n) M(x) \nabla \psi_n$ is bounded in $\elleomT1$;
        \item  $T_n(f(x,t))$ is bounded in $\elleomT m$.
    \end{itemize}
    This implies that the sequence $((u_n)_t)_n$ in bounded in $\elle1(0,T;W^{-1,1}(\Omega))$.
    Thus, we can apply \cite[Corollary 4]{Simon} to obtain that the sequence $(u_n)_n$ has a converging subsequence in $\elleomT {m^{\star\star}}$. Let $u$ be the limit of such subsequence: we have $u\in\elle\infty(0,T;\elleom{\frac{Nm}{N+2-2m}})$, $u\in\elle {m^\star}(0,T;\Sobom{m^\star})$ (since $(u_n)_n$ is bounded in $\elle {m^\star}(0,T;\Sobom{m^\star})$ and since $\nabla u_n\rightharpoonup\nabla u$), and $u\in\elleomT {m^{\star\star}}$ by Fatou's Lemma and Lemma \ref{lemma:stima_m**_unif:infinite_energy}. Moreover, $u_t\in \elle1(0,T;W^{-1,1}(\Omega))$.

    By Lemma \ref{lemma:psin_bounded_unif}, there exists a subsequence of $(\psi_n)_n$ that strongly converges in $\elleomT2$ and weakly converges in $\elle2(0,T;\Hom)$ to some function $\psi(x,t)$ which is a weak solution to
    \[
    -\operatorname{div}(M(x)\nabla\psi)=u^\theta(x,t)\qquad \text{in }\Omega
    \]
    for every $t\in (0,T)$, with zero Dirichlet boundary conditions.
    Moreover, since $u_n^\theta$ is strongly convergent in $\elleomT p$ with $p>\frac{N}{2}$, by applying standard elliptic estimates \cite{BoccardoCroce} to the equation
    \[
    -\operatorname{div}(M(x)\nabla(\psi_n-\psi))=u_n^\theta-u^\theta,
    \]
    we get (up to subsequences):
    \begin{itemize}
        \item $\psi_n\to\psi$ in $\elle\infty(0,T;\Hom)$, which implies that $\psi\in C^0([0,T];\Hom)$; 
        \item $\psi_n\to\psi$ in $\elle\infty(\Omega_T)$, which implies that $\psi\in \elle\infty(\Omega_T)$.
    \end{itemize}

    Let $\varphi$ and $v$ be test functions as in the statement of Theorem \ref{Main_Theorem_2}. Then, we take the limit (up to subsequences) as $n\to\infty$ in 
    \begin{equation}
        \begin{dcases}
                \int_0^T \langle (u_n)_t,\varphi\rangle+\ioT A(x,t)\nabla u_n\cdot \nabla \varphi = \ioT T_n(u_n) M(x)\nabla \psi_n\cdot\nabla\varphi + \ioT T_n(f)\varphi\\
                \io M(x)\nabla \psi_n(t)\cdot\nabla v = \io u_n(t)^\theta v,
        \end{dcases}
    \end{equation}
    to conclude that the couple $(u,\psi)$ is a solution to \eqref{eq:KS-Sys-ParabEll} in the sense \eqref{DefSol_parabell_theo3}.
\end{proof}

\subsection{Existence of entropy solutions}\label{sec:entropy}

Since $\frac{2N+4}{N+6}>1$ if and only if $N>2$, which is our assumption, the techniques we have employed up until now fail for any $m\in [1,\frac{2N+4}{N+6})$. 
More precisely, since the terms $T_n(u_n)M(x)\nabla\psi_n$ are not bounded in $\elleomT1$, it is impossible to pass to the limit as $n\to+\infty$ in 
\[
\ioT T_n(u_n)M(x)\nabla\psi_n\cdot\nabla\varphi
\]
to prove the existence of a distributional solution to \eqref{eq:KS-Sys-ParabEll}. 
Instead, we prove the existence of an entropy solution (see Definition \ref{Def:entropysol}) to \eqref{eq:KS-Sys-ParabEll} system with any datum $f\in\elleomT1$.

In the next lemma, we show that, even though for data $f$ not belonging to $\elleomT{\frac{2N+4}{N+4}}$ one should not expect the sequence of approximate solutions $(u_n)_n$ to be bounded in $\elle2(0,T;\Hom)$, the sequence $(T_k(u_n))_n$ given by their truncations at any height $k>0$ is, indeed, bounded in $\elle2(0,T;\Hom)$.

\begin{lemma}\label{lemma:Bound_Troncate}
    Let $k\geq 0$ and $f\in\elleomT1$. Then, the sequence $(T_k(u_n))_n$ is bounded in $\elle2(0,T;\Hom)$.
\end{lemma}
\begin{proof}
    Let us choose $T_k(u_n)$ as a test function in the first equation of \eqref{eq:approx_KS_system}
    $$
    \int_0^t \langle (u_n)_t,T_k(u_n)\rangle+\alpha\iot\abs{\nabla T_k(u_n)}^2\leq \iot T_n(u_n) M(x)\nabla\psi_n\cdot\nabla T_k(u_n) +\iot f(x,t)T_k(u_n).
    $$
    Applying Young's inequality to the right hand side (note that $\abs{T_n(u_n)}\leq k$ where $\nabla T_k(u_n)\neq0$) we obtain
    $$
    \io\Theta_k(u_n)+\alpha\iot\abs{\nabla T_k(u_n)}^2
    \leq 
    k^2\frac{\beta^2}{2\alpha}\ioT \abs{\nabla\psi_n}^2+\frac{\alpha}{2}\iot\abs{\nabla T_k(u_n)}^2 +k\ioT f(x,t).
    $$
    Here, as before, $\Theta_k$ denotes the primitive of $T_k$ such that $\Theta_k(0)=0$.
    After simplifying the gradient term and dropping the first (positive) term, we end up with
    $$
    \frac{\alpha}{2}\io\abs{\nabla T_k(u_n)}^2
    \leq 
    \frac{k^2\beta^2T}{2\alpha}\norma{\psi_n}{\elle\infty(0,T;\Hom)}^2+k\norma{f}{\elleomT1}.
    $$
    Which, in view of Lemma \ref{lemma:psin_bounded_unif}, concludes the proof.    
\end{proof}

Notice that Lemma \ref{lemma:stima_m**_unif:infinite_energy} does not include the case $f\in\elleomT1$. The following result, whose proof follows \cite{BO-Kellersegel}, gives an a priori estimate on the sequence $(u_n)_n$ in this case.

\begin{lemma}\label{lemma:stima1**}
    Let $f\in\elleomT1$, then the sequence of approximating solutions $(u_n)_n$ is bounded in $\elleomT p$ for every $p<\frac{N+2}{N}=1^{\star\star}$ and the sequence of gradients $(\nabla u_n)_n$ is bounded in $\elleomT q$ for every $q<\frac{N+2}{N+1}=1^{\star}$.
\end{lemma}
\begin{proof}
    Let $1<\lambda<2$ and choose $\frac{1-(1+u_n)^{1-\lambda}}{\lambda-1}$ as a test function in the second equation of \eqref{eq:approx_KS_system}. Repeating the same steps as in the proof of Lemma \ref{lemma:stima_m**_unif:infinite_energy}, we define the function
    $$F_n(s)=\int_0^s (\sigma+1)^{-\lambda}T_n(\sigma) \mathrm{d}\sigma$$
    and choose $F_n(u_n)$ as a test function in the second equation of \eqref{eq:approx_KS_system} to get
    $$
    \io (u_n+1)^{-\lambda}T_n(u_n) M(x)\nabla \psi_n\cdot\nabla u_n 
    =
    \io M(x)\nabla \psi_n\cdot\nabla F_n(u_n)
    =
    \io T_n(u_n)^\theta F_n(u_n)
    $$
    $$
    \leq \io T_n(u_n)^\theta \frac{(u_n+1)^{2-\lambda}-1}{2-\lambda}
    \leq
    \io u_n^{\theta} \frac{(u_n+1)^{2-\lambda}-1}{2-\lambda}
    $$
    Now choose 
    $$g(u_n)=\frac{1-(1+u_n)^{1-\lambda}}{\lambda-1}$$
    (which is bounded from above by $\frac{1}{\lambda-1}$) as a test function in the first equation of \eqref{eq:approx_KS_system} to get
\begin{equation}\label{eq:proof:stima1**:eq0}
    \begin{split}
            \langle (u_n)_t,g(u_n)\rangle&
    +\io(u_n+1)^{-\lambda}A(x,t)\nabla u_n \cdot \nabla u_n\\&
    \leq
    \io (u_n+1)^{-\lambda}T_n(u_n) M(x)\nabla \psi_n\cdot\nabla u_n 
    +\io f(x,t)g(u_n).
    \end{split}
\end{equation}
    Which implies
    \begin{equation}\label{EQ_stimagrad_Parabell_fL1}
        \langle (u_n)_t,g(u_n)\rangle+\alpha\io\frac{\abs{\nabla u_n}^2}{(u_n+1)^{\lambda}}
        \leq
        \io u_n^{\theta} \frac{(u_n+1)^{2-\lambda}-1}{2-\lambda}+\frac{1}{\lambda-1}\io f(x,t).        
    \end{equation}
    We now focus on the term
    $$\io u_n^{\theta} \frac{(u_n+1)^{2-\lambda}-1}{2-\lambda}.$$
    Let $k\geq0$ and apply the H\"older inequality (with exponent $\frac{N}{2}$) along with Lemma \ref{StimaL1ParabEll} to get
    \begin{equation}\label{eq:proof:stima1**:eq1}
        \begin{split}
            \io u_n^{\theta} &\frac{(u_n+1)^{2-\lambda}-1}{2-\lambda}
    \leq \int_{u_n(t)\leq k} u_n^{\theta} \frac{(u_n+1)^{2-\lambda}-1}{2-\lambda} + 
    \int_{u_n(t)> k} u_n^{\theta} \frac{(u_n+1)^{2-\lambda}-1}{2-\lambda}\\
    &\leq
    k^{\theta}\frac{(k+1)^{2-\lambda}}{2-\lambda}\abs{\Omega}+\frac{1}{2-\lambda}
    \left(\frac{\norma{f}{\elleomT1}}{k^{1-\theta\frac{N}{2}}}\right)^\frac{2}{N}
    \left(\io \left((u_n(t)+1)^{2-\lambda}-1\right)^{\frac{N}{N-2}}\right)^{\frac{N-2}{N}}.
        \end{split}
    \end{equation}
    
    Since, for every $\eta\geq0$, the following inequality holds
    $$
    \left[(1+\eta)^{2-\lambda}-1\right]^{\frac{1}{2}} \leq(1+\eta)^{\frac{2-\lambda}{2}}=\left[(1+\eta)^{\frac{2-\lambda}{2}}-1\right]+1,
    $$
    we have
    $$
    \io\left[\left(1+u_n\right)^{2-\lambda}-1\right]^{\frac{N}{N-2}} 
    \leq 
    \io\left\{\left[\left(1+u_n\right)^{\frac{2-\lambda}{2}}-1\right]+1\right\}^{{\frac{2N}{N-2}}} 
    \leq 
    C \left(1+\io\left[\left(1+u_n\right)^{\frac{2-\lambda}{2}}-1\right]^{{\frac{2N}{N-2}}}\right),
    $$
    for some $C>0$.
    It follows that
        \begin{equation}\label{eq:proof:stima1**:eq3}
        \begin{split}
            \io u_n^{\theta} \left((u_n+1)^{2-\lambda}-1\right)
    \leq 
    k^{\theta+2-\lambda}\abs{\Omega}
    +    \frac{C}{2-\lambda}
    \left(\frac{\norma{f}{\elleomT1}}{k^{1-\theta\frac{N}{2}}}\right)^\frac{2}{N}
    \left(1+\left[\io\left[\left(1+u_n\right)^{\frac{2-\lambda}{2}}-1\right]^{{\frac{2N}{N-2}}}\right]^{\frac{N-2}{N}}\right).
        \end{split}
    \end{equation}

    On the other hand, by the Sobolev inequality, we have
    \begin{equation}\label{eq:proof:stima1**:eq2}
        \begin{split}
            \io \frac{\left|\nabla u_n\right|^2}{\left(1+u_n\right)^\lambda}=
    \frac{4}{(2-\lambda)^2} \int_{\Omega}\left|\nabla\left[\left(1+u_n\right)^{\frac{2-\lambda}{2}}-1\right]\right|^2 \geq 
    \frac{4 \mathcal{S}}{(2-\lambda)^2}\left[\int_{\Omega}\left[\left(1+u_n\right)^{\frac{2-\lambda}{2}}-1\right]^{{\frac{2N}{N-2}}}\right]^{\frac{N-2}{N}}.
        \end{split}
    \end{equation}
    
    Notice that, if we define $\mathcal{G}(s)=\int_0^s g(\sigma)\mathrm{d}\sigma$, the term on the left-hand side of \eqref{EQ_stimagrad_Parabell_fL1} becomes $\frac{d}{dt}\io \mathcal{G}(u_n)$. Thus, putting together \eqref{EQ_stimagrad_Parabell_fL1}, \eqref{eq:proof:stima1**:eq3} and \eqref{eq:proof:stima1**:eq2}, we obtain
    \begin{equation}
    \begin{split}
    \frac{d}{dt}\io \mathcal{G}(u_n) +
    &\left(\alpha -C \frac{(2-\lambda)^2}{4\mathcal{S}}\frac{\norma{f}{\elleomT1}^{\frac{2}{N}}}{(2-\lambda) k^{\frac{2}{N}-\theta}}\right)
    \io \frac{\left|\nabla u_n\right|^2}{\left(1+u_n\right)^\lambda} \\
    & \leq 
    \frac{1}{\lambda-1}\io f(x,t) +
    k^{\theta}\frac{(k+1)^{2-\lambda}}{2-\lambda}\abs{\Omega}
    +\frac{C}{2-\lambda} \frac{\norma{f}{\elleomT 1} ^{\frac{2}{N}}}{k^{\frac{2}{N}-\theta}} .
    \end{split}
    \end{equation}
    Choosing $k=k_0$ such that 
    $$\alpha -C \frac{(2-\lambda)^2}{4\mathcal{S}}\frac{\norma{f}{\elleomT1}^{\frac{2}{N}}}{(2-\lambda) k_0^{\frac{2}{N}-\theta}}=\frac{\alpha}{2},$$
    we have
    $$\frac{d}{dt}\io \mathcal{G}(u_n) +
    \frac{\alpha}{2}\io \frac{\left|\nabla u_n\right|^2}{\left(1+u_n\right)^\lambda}
    \leq 
    \frac{1}{\lambda-1}\io f(x,t) +
    k_0^{\theta}\frac{(k_0+1)^{2-\lambda}}{2-\lambda}\abs{\Omega}
    +\frac{C}{2-\lambda} \frac{\norma{f}{\elleomT 1} ^{\frac{2}{N}}}{k_0^{\frac{2}{N}-\theta}} .
    $$
    Integrating from $0$ to $t$, we obtain
    $$\io \mathcal{G}(u_n)(t) +
    \frac{\alpha}{2}\iot \frac{\left|\nabla u_n\right|^2}{\left(1+u_n\right)^\lambda}
    \leq 
    \frac{1}{\lambda-1}\ioT f(x,t) +
    k_0^{\theta}\frac{(k_0+1)^{2-\lambda}}{2-\lambda}\abs{\Omega}T
    +\frac{C}{2-\lambda} \frac{\norma{f}{\elleomT 1} ^{\frac{2}{N}}}{k_0^{\frac{2}{N}-\theta}}T,
    $$
    which is
    \begin{equation}\label{eq:proof:stima1**:eq3.5}
        \io \mathcal{G}(u_n)(t) +
    \frac{\alpha}{2}\iot \frac{\left|\nabla u_n\right|^2}{\left(1+u_n\right)^\lambda}
    \leq C\left(\|f\|_{L^1(\Omega_T)}, \lambda\right),
    \end{equation}
    where $C\left(\|f\|_{L^1(\Omega_T)},\lambda\right)$ is a positive constant which tends to infinity as $\lambda\to1^+$.

    Since $\mathcal{G}(s)=s+(1+s)^{2-\lambda}-1$ (which, for $s\geq0$, is greater than $s^{2-\lambda}$) and $u_n\geq0$, we have, up to increasing the constant $C\left(\|f\|_{L^1(\Omega_T)},\lambda\right)$,
    \begin{equation}\label{DoppiaStima-Lemma-StimaconfL1}
        \left\{\begin{aligned}
        & \norma{(u_n+1)^{2-\lambda}-1}{\elle\infty(0,T;\elleom1)}\leq C\left(\|f\|_{L^1(\Omega_T)}, \lambda\right)\\
        & \norma{\nabla \left((u_n+1)^\frac{2-\lambda}{2}-1\right)}{\elleomT2}^2\leq C\left(\|f\|_{L^1(\Omega_T)}, \lambda\right).
        \end{aligned}\right.
    \end{equation}
    Recall that, by Proposition \eqref{GNSparabTheo}, there exists $C=C(N,\Omega)>0$ such that
    $$\norma{(u_n+1)^\frac{2-\lambda}{2}-1}{\elleomT{2\frac{N+2}{N}}}
    \leq\norma{(u_n+1)^\frac{2-\lambda}{2}-1}{\elle\infty(0,T;\elleom2)}^\frac{4}{N}
    \norma{\nabla \left((u_n+1)^\frac{2-\lambda}{2}-1\right)}{\elleomT2}^2.
    $$
    Moreover, since $\frac{2-\lambda}{2}\in (0,1)$ and $u_n\geq0$ we have, by concavity,
    $$(u_n+1)^\frac{2-\lambda}{2}-1\leq u_n^\frac{2-\lambda}{2},$$
    which implies
    $$
    \norma{(u_n+1)^\frac{2-\lambda}{2}-1}{\elle\infty(0,T;\elleom2)}
    \leq 
    \norma{u_n^\frac{2-\lambda}{2}}{\elle\infty(0,T;\elleom2)}
    \leq
    \norma{u_n^{2-\lambda}}{\elle\infty(0,T;\elleom1)}^\frac{1}{2}.$$
    By convexity, since $2-\lambda>1$, we get
    $$\norma{(u_n+1)^\frac{2-\lambda}{2}-1}{\elle\infty(0,T;\elleom2)}
    \leq 
    \norma{u_n^{2-\lambda}}{\elle\infty(0,T;\elleom1)}^\frac{1}{2}
    \leq
    \norma{(u_n+1)^{2-\lambda}-1}{\elle\infty(0,T;\elleom1)}^\frac{1}{2}.$$
    To sum up, we have proved that, for any $\lambda\in(1,2)$, there exists a positive constant $C=C\left(\norma{f}{\elleomT1},\lambda,N,\Omega\right)$ such that
    \begin{equation}\label{eq:proof:stima1**:eq4}
        \norma{(u_n+1)^\frac{2-\lambda}{2}-1}{\elleomT{2\frac{N+2}{N}}}
    \leq C,
    \end{equation}
    which implies that the sequence $(u_n^\frac{2-\lambda}{2})_n$ is bounded in $\elleomT{2\frac{N+2}{N}}$. This means that $(u_n)_n$ is bounded in $\elleomT{{2-\lambda}\frac{N+2}{N}}$ for any $\lambda\in(1,2)$, i.e. $(u_n)_n$ is bounded in $\elleomT p$ for any $p<\frac{N+2}{N}$.

    We now follow the argument in the proof of Lemma \ref{lemma:u_nW1m*_KS_parabell}. Choosing $q<2$ and using the H\"older inequality, we obtain 
    $$\ioT \abs{\nabla u_n}^q =
    \ioT \frac{\abs{\nabla u_n}^q}{(1+u_n)^{\frac{q\lambda}{2}}}(1+u_n)^{\frac{q\lambda}{2}} 
    \leq
    \left(\ioT \frac{\abs{\nabla u_n}^2}{(1+u_n)^{\lambda}}\right)^\frac{q}{2}
    \left(\ioT (1+u_n)^{\frac{q\lambda}{2-q}}\right)^{1-\frac{q}{2}},$$
    thus, by \eqref{eq:proof:stima1**:eq3.5}, we get
    $$\ioT \abs{\nabla u_n}^q\leq 
     \frac{2\,C\left(\norma{f}{\elleomT1}, \lambda\right)}{\alpha}
    \left(\ioT (1+u_n)^{\frac{q\lambda}{2-q}}\right)^{1-\frac{q}{2}}.$$

    Setting $\frac{q\lambda}{2-q}={2-\lambda}\frac{N+2}{N}$ and recalling \eqref{eq:proof:stima1**:eq4}, we conclude that $(\nabla u_n)_n$ is bounded in $\elleomT q$ with
    $$q=q(\lambda)=\frac{2(2-\lambda)(N+2)}{\lambda N +(2-\lambda)(N+2)}.$$
    Note that $q(\lambda)$ is always smaller that $\frac{N+2}{N+1}$ and tends to $\frac{N+2}{N+1}$ as $\lambda\to1$. This concludes the proof.
\end{proof}

{To prove the existence of an entropy solution, it is necessary to pass to the limit in the integral
$$\ioT T_n(u_n)M(x)\nabla\psi_n \nabla T_k (u_n-\varphi).$$
To do so, the weak convergence of the sequence $(u_n)_n$ is not sufficient. Thus, we employ a technique of \cite{PorrettaExistenceParabolic} to prove the convergence of smooth truncations of $(u_n)_n$. This, in turn, implies the strong convergence of $(u_n)_n$ in $\elleomT1$.}

\begin{lemma}\label{lemma:ae_convergence_un}
    The sequence $(u_n)_n$ converges, up to subsequences, almost--everywhere in $\Omega_T$.
\end{lemma}
\begin{proof}
    Consider a smooth truncation function $\mathcal{T}_k(s)$ such that:
    \begin{itemize}
        \item $\mathcal T_k(s)=s$ whenever $\abs{s}\leq \frac{k}{2}$;
        \item $\mathcal T_k(s)=k$ whenever $\abs{s}\geq k$;
        \item $0\leq \mathcal{T}_k'(s)\leq 2\rchi_{[-k,k]}(s).$
    \end{itemize}
    Then, choose $\mathcal{T}'_k(u_n)\varphi$, with $\varphi$ smooth, as a test function in the first equation of \eqref{eq:approx_KS_system} to obtain
    \begin{equation}
        \begin{split}   
            \int_0^T\langle (\mathcal{T}_k(u_n))_t,\varphi\rangle +& \ioT A(x,t)\nabla u_n \mathcal{T}'_k(u_n)\nabla \varphi +\ioT A(x,t)\nabla u_n \mathcal{T}''_k(u_n)\nabla u_n \varphi\\
            = &\ioT T_n(u_n) M(x)\nabla\psi_n \mathcal{T}'_k(u_n)\nabla \varphi \\&+ \ioT T_n(u_n) M(x)\nabla\psi_n \mathcal{T}''_k(u_n)\nabla u_n \varphi + \ioT T_n(f)\mathcal{T}'_k(u_n)\varphi,
        \end{split}
    \end{equation}
    which can be written, in the sense of distributions, as
    \begin{equation}
        \begin{split}
            (\mathcal{T}_k(u_n))_t -&\operatorname{div}( A(x,t)\nabla\mathcal{T}_k(u_n))+A(x,t)\nabla u_n\cdot \nabla u_n \mathcal{T}''_k(u_n) \\=
            -\operatorname{div}&(T_n(u_n) M(x)\nabla\psi_n \mathcal{T}'_k(u_n))\\&+ T_n(u_n) M(x)\nabla\psi_n \cdot \nabla u_n  \mathcal{T}''_k(u_n) + T_n(f)\mathcal{T}'_k(u_n).
        \end{split}
    \end{equation}
    By the assumption on $\mathcal{T}'_k(s)$, the term $T_n(u_n)\mathcal{T}'_k(u_n)$ is bounded in $\elleomT\infty$ and $\abs{\mathcal{T}'_k(u_n)}\leq2\rchi_{\{\abs{u_n}\leq k\}}$. This estimate, along with Lemma \ref{lemma:Bound_Troncate} and Lemma \ref{lemma:psin_bounded_unif} (specifically, the boundedness of $(\psi_n)_n$ in $\elle2(0,T;\Hom)$), implies that $(\mathcal{T}_k(u_n))_t$ is bounded in $\elle1(0,T;W^{-1,1}(\Omega))$ with respect to $n$. At the same time, Lemma \ref{lemma:Bound_Troncate} implies that the sequence $(\mathcal{T}_k(u_n))_n$ is bounded in $\elle2(0,T;\Hom)$ with respect to $n$. We can thus apply \cite[Corollary 4]{Simon} to conclude that $(\mathcal{T}_k(u_n))_n$ is compact in $\elleomT2$.
    It follows that, for every $k\geq0$ there exists a subsequence of $(\mathcal{T}_k(u_n))_n$ which converges in $\elleomT2$ and almost--everywhere.

    We can now prove that the sequence $(u_n)_n$ is Cauchy in measure. Let $\lambda>0$ and notice that
    $$\abs{\left\{ \abs{u_n-u_m}\geq \lambda \right\}}\leq \abs{\left\{ \abs{u_n}\geq \frac{k}{2} \right\}}+\abs{\left\{ \abs{u_m}\geq \frac{k}{2} \right\}}+\abs{\left\{ \abs{\mathcal {T}_k (u_n)-\mathcal {T}_k (u_m)}\geq \lambda \right\}}.$$
    Since the sequence $(u_n)_n$ is bounded in some Lebesgue space, the first two terms on the right-hand side are arbitrarily small choosing $k$ large enough, while the last term tends to zero as $n$ and $m$ tend to infinity for any chosen $k$. It follows that the sequence $(u_n)_n$ converges in measure, which implies the almost--everywhere convergence up to subsequences.
\end{proof}

\begin{cor}\label{cor:strongconvL1}
    Let $u$ be a weak limit of $(u_n)_n$ in $\elleomT s$ for some $s>1$. Then, up to subsequences, $u_n\to u$ strongly in $\elleomT1$.
\end{cor}

We can now prove Theorem \ref{Main_Theorem_4}. 
\begin{proof}[Proof of Theorem \ref{Main_Theorem_4}]
We start by proving the existence of an entropy solution $(u,\psi)$ to \eqref{eq:KS-Sys-ParabEll}. 
    Let $k\geq 0$ and $\varphi\in\elleomT\infty\cap\elle2(0,T;\Hom)$ such that $\varphi_t\in \elleomT 1+ \elle2(0,T;H^{-1}(\Omega))$ and $\varphi(0)=0$. We choose $T_k(u_n-\varphi)$ as a test function in the first equation of \eqref{eq:approx_KS_system} to obtain
    \begin{multline}\label{Passaggio1-EsisteEntropy-KS-Parabell}
        \int_0^T\langle (u_n)_t,T_k(u_n-\varphi)\rangle + \iot A(x,s)\nabla u_n\cdot \nabla T_k(u_n-\varphi) \\
        =\iot T_n(u_n) M(x)\nabla \psi_n\cdot\nabla T_k(u_n-\varphi) + \iot T_n(f)T_k(u_n-\varphi).
    \end{multline}
    We will now examine each term of \eqref{Passaggio1-EsisteEntropy-KS-Parabell}. First, notice that 
    $$\int_0^t\langle (u_n)_t,T_k(u_n-\varphi)\rangle=\io \Theta_k(u_n-\varphi)(t)+\int_0^t\langle \varphi_t,T_k(u_n-\varphi)\rangle.$$
    Then, by the almost--everywhere convergence of the sequence $(u_n)_n$, we have that $\Theta_k(u_n-\varphi)(t)\to\Theta_k(u-\varphi)(t)$ almost--everywhere (which allows to to apply Fatou's Lemma, since $\Theta_k$ is a positive function). Also, notice that
    \[\int_0^t\langle \varphi_t,T_k(u_n-\varphi)\rangle\]
    passes to the limit as $n\to+\infty$ by the weak convergence in $\elle2(0,T;\Hom)$ and the $*-$weak convergence in $\elleomT\infty$ of $T_k(u_n-\varphi)$ to $T_k(u-\varphi)$.

    In the second term of \eqref{Passaggio1-EsisteEntropy-KS-Parabell}, we proceed similarly:
    $$
    \iot A(x,s)\nabla u_n\cdot \nabla T_k(u_n-\varphi) = \iot A(x,s)\nabla (u_n-\varphi)\cdot \nabla T_k(u_n-\varphi) + \iot A(x,s)\nabla \varphi\cdot \nabla T_k(u_n-\varphi).
    $$
    Notice that every integral in well-defined, since the right-hand side is equal to 
    $$
    \iot A(x,s)\nabla T_k(u_n-\varphi)\cdot \nabla T_k(u_n-\varphi) + \iot A(x,s)\nabla \varphi\cdot \nabla T_k(u_n-\varphi).
    $$
    This allows us to take the liminf in the first term by weak lower semicontinuity of the quadratic form induced by $A(x,t)$ on $\elle2(0,T;\Hom)$ and the limit (by weak convergence) in the second term.

    For the third term of \eqref{Passaggio1-EsisteEntropy-KS-Parabell}, we have 
    $$\iot T_n(u_n) M(x)\nabla \psi_n\cdot\nabla T_k(u_n-\varphi)=\int_{\Omega_t\cap\{\abs{u_n-\varphi}\leq k\}} T_n(u_n) M(x)\nabla \psi_n\cdot\nabla T_k(u_n-\varphi).$$
    Notice that, since $\varphi\in\elleomT\infty$, the integral is computed only where $$\abs{u_n}\leq k+\norma{\varphi}{\elleomT\infty}=:M.$$
    Thus, if $n\geq M$, we can write it as
    $$\int_{\Omega_t\cap\{\abs{u_n-\varphi}\leq k\}} T_n(T_M(u_n)) M(x)\nabla \psi_n\cdot\nabla T_k(u_n-\varphi)=
    \iot T_M(u_n) M(x)\nabla \psi_n\cdot\nabla T_k(u_n-\varphi).$$
    Note that, by Vitali's theorem and Corollary \ref{cor:strongconvL1}, the sequence $u_n^\theta$ is strongly convergent in $\elleomT p$, with $p=\frac1\theta>\frac{N}{2}$. Thus, by reasoning as in the proof of Theorem \ref{Main_Theorem_2}, we prove that (up to subsequences):
    \begin{itemize}
        \item $\psi_n\to\psi$ in $\elle\infty(0,T;\Hom)$ (which also implies that $\psi\in C^0([0,T];\Hom)$);
        \item $\psi_n\to\psi$ in $\elle\infty(\Omega_T)$ (which also implies that $\psi\in \elle\infty(\Omega_T)$).
    \end{itemize}
    It follows that
    $$ T_M(u_n) M(x)\nabla \psi_n\to T_M(u)M(x)\nabla\psi\quad\text{ in } \left(\elleomT 2\right)^N,$$
    which allows us to pass to the limit as $n\to\infty$ in the third term of \eqref{Passaggio1-EsisteEntropy-KS-Parabell}. The last term of \eqref{Passaggio1-EsisteEntropy-KS-Parabell} poses no difficulties in the passage to the limit.

    We can now pass to the limit as $n\to+\infty$ along a suitable subsequence to conclude that, for almost every $t\in(0,T)$, we have
    \begin{multline}
        \io \Theta_k(u-\varphi)(t)+\int_0^t \langle (\varphi)_t,T_k(u-\varphi)\rangle+
        \iot A(x,s)\nabla u\cdot \nabla T_k(u-\varphi)
        \\\quad  \leq\iot u M(x)\nabla \psi\cdot\nabla T_k(u-\varphi) + \iot f(x,s)T_k(u-\varphi),
    \end{multline}
    which, along with Lemma \ref{lemma:Bound_Troncate}, proves that $(u,\psi)$ is an entropy solution to \eqref{eq:KS-Sys-ParabEll}. To conclude, we observe that the bounds on the sequence $(u_n)_n$ in Lemmata \ref{StimaL1ParabEll}, \ref{lemma:stima_m**_unif:infinite_energy}, \ref{lemma:u_nW1m*_KS_parabell} and \ref{lemma:stima1**} also hold for $u$.
\end{proof}

\section*{Acknowledgements}

I would like to thank Luigi Orsina for his meticulous and insightful reading of this manuscript, and the anonymous reviewer for their careful examination of all the details, which greatly contributed to improving the clarity and accuracy of the paper. I am a member of the GNAMPA group of INdAM.

\end{document}